\documentclass[12pt]{article}

\usepackage{amssymb,amsmath,amsthm,sectsty,url,mathrsfs,graphicx}
\usepackage[letterpaper,hmargin=1.0in,vmargin=1.0in]{geometry}
\usepackage[colorlinks,linkcolor=black,citecolor=black,filecolor=black,urlcolor=black]{hyperref}
\usepackage{color}

\sectionfont{\large}
\subsectionfont{\normalsize}
\numberwithin{equation}{section}

\newtheorem{maintheorem}{Theorem}
\newtheorem{theorem}{Theorem}[section]
\newtheorem{lemma}[theorem]{Lemma}
\newtheorem{proposition}[theorem]{Proposition}
\newtheorem{corollary}[theorem]{Corollary}
\newtheorem{fact}[theorem]{Fact}
\newtheorem{definition}[theorem]{Definition}
\newtheorem{remark}[theorem]{Remark}

\newtheorem{example}[theorem]{Example}

\def\E{\mathop{\mathbb E}}
\newcommand{\Var}{{\bf Var}}

\newcommand{\ct}{\mathrm{ct}}
\renewcommand{\Pr}{ \mathrm P}

\newcommand{ \rel}{ t_{\mathrm{rel}} }
\newcommand{ \mix}{ t_{\mathrm{mix}} }

\newcommand{ \hit}{ \mathrm{hit} }
\newcommand{ \h}{ \mathrm{H} }

\newcommand{\N}{\mathbb N}
\newcommand{\R}{\mathbb R}

\usepackage[normalem]{ulem}

\begin{document}

\title{Characterization of cutoff for reversible Markov chains}
\author{Riddhipratim Basu
\thanks{
Department of Statistics, UC Berkeley, California, USA. E-mail: {\tt riddhipratim@stat.berkeley.edu}. Supported by UC Berkeley Graduate Fellowship.}
\and Jonathan Hermon
\thanks{
Department of Statistics, UC Berkeley, USA. E-mail: {\tt jonathan.hermon@stat.berkeley.edu}.}
\and
Yuval Peres
\thanks{
Microsoft Research, Redmond, Washington, USA. E-mail: {\tt peres@microsoft.com}.}
}
\date{}
\maketitle

\begin{abstract}
A sequence of Markov chains is said to exhibit (total variation) cutoff if the convergence to stationarity in total variation distance is abrupt. We consider reversible lazy chains. We prove a necessary and sufficient condition for the occurrence of the cutoff phenomena in terms of concentration of hitting time of ``worst" (in some sense) sets of stationary measure at least $\alpha$, for some $\alpha \in (0,1)$.

We also give general bounds on the total variation distance of a reversible chain at time $t$ in terms of the probability that some ``worst" set of stationary measure at least $\alpha$ was not hit by time $t$.
As an application of our techniques we show that a sequence of lazy Markov chains on finite trees exhibits a cutoff iff the product of their spectral gaps and their (lazy) mixing-times tends to $\infty$.
\end{abstract}

\newpage


\section{Introduction}
In many randomized algorithms, the mixing-time of an underlying Markov chain is the main component of the running-time (see \cite{sinclair1993algorithms}). We obtain a tight bound on $t_{\mathrm{mix}}(\epsilon)$ (up
to an absolute
constant independent of $\epsilon$) for lazy reversible Markov chains in terms of hitting times of large sets (Proposition \ref{prop: TVbound0}, (\ref{eq: TVbound0})). This refines previous results
in the same spirit  (\cite{peres2011mixing} and
\cite{oliveira2012mixing}, see related work), which gave a less precise
characterization of the mixing-time in terms of hitting-times (and were restricted to hitting times of sets
whose stationary measure is at most 1/2).

Loosely speaking, the (total variation) \emph{\textbf{cutoff phenomenon}}
 occurs when over a negligible period of time, known as the \emph{\textbf{cutoff
window}}, the (worst-case) total variation distance (of a certain finite
Markov chain from its stationary distribution) drops abruptly from a value
close to 1 to near $0$. In other words, one should run the chain until
the cutoff point for it to even slightly mix in total
variation, whereas running it any further  is essentially redundant.

Though many families of chains are believed to exhibit cutoff, proving the
occurrence of this phenomenon is often an extremely challenging task. The cutoff phenomenon was given its name by Aldous and Diaconis
in their seminal paper $\cite{aldous1986shuffling}$ from 1986 in which they
suggested the following
open problem (re-iterated in \cite{diaconis1996cutoff}), which they refer
to as ``the most interesting problem": 
``\textit{Find abstract
conditions which ensure that the cutoff phenomenon occurs}". Although
drawing much attention, the progress made in the investigation
of the cutoff phenomenon has been mostly done through understanding examples and the field suffers from a rather disturbing lack of general theory.  Our bound on the mixing-time is sufficiently sharp to imply a characterization of  cutoff  for reversible Markov chains in terms
of concentration of hitting times.

We use our general characterization of cutoff to give a sharp spectral condition for cutoff in lazy weighted nearest-neighbor random walks on trees (Theorem \ref{thm: treescutoff}).

Generically, we shall denote the state space of a Markov chain by $\Omega
$ and its stationary distribution by $\pi$ (or $\Omega_n$ and $\pi_n$, respectively, for the $n$-th chain in a sequence of chains). Let $(X_t)_{t=0}^{\infty}$ be an irreducible Markov chain on a finite state
space $\Omega$ with transition matrix $P$ and stationary distribution $\pi$. We denote such a chain by $(\Omega,P,\pi)$. We say that the chain is finite,
whenever $\Omega$ is finite. We say the chain is \emph{\textbf{reversible}}
if $\pi(x)P(x,y)=\pi(y)P(y,x)$, for any $x,y \in \Omega$.

We call a chain
\emph{\textbf{lazy}} if $P(x,x) \ge 1/2$, for all $x$. In this paper, all discrete-time chains would be assumed to be lazy, unless
otherwise is specified. To avoid periodicity and near-periodicity issues one often considers the lazy version of the chain, defined by replacing $P$ with $P_L:=(P+I)/2$. Another way to avoid
periodicity issues is to consider the continuous-time
version of the chain, $(X_t^{\mathrm{ct}})_{t
\ge 0}$,
which is a continuous-time Markov chain whose heat kernel is defined by $H_t(x,y):=\sum_{k=o}^{\infty}\frac{e^{-t}t^k}{k!}P^t(x,y)$.   

We denote by $\Pr_{\mu}^t$ ($\Pr_{\mu}$) the distribution of $X_t$ (resp.~$(X_t)_{t \ge 0}$), given that the initial distribution is $\mu$. We denote by $\h_{\mu}^t$
($\h_{\mu}$) the distribution of $X_t^{\mathrm{ct}}$ (resp.~$(X_t^{\mathrm{ct}})_{t
\ge 0}$), given that the initial distribution is $\mu$.  When $\mu=\delta_x$, the
  Dirac measure on some $x \in \Omega$ (i.e$.$ the chain starts at $x$ with probability 1), we simply write $\Pr_x^t$ ($\Pr_x$) and $\h_x^t$ ($\h_x$). For any $x,y \in \Omega$ and $t \in \N$ we write $\Pr_x^{t}(y):=\Pr_{x}(X_t=y)=P^{t}(x,y)$.

 We denote the set of probability distributions on a (finite) set $B$ by
$\mathscr{P}(B) $.  For any $\mu,\nu \in \mathscr{P}(B)$,  their \emph{\textbf{total-variation
distance}} is defined to
be
$\|\mu-\nu\|_\mathrm{TV} := \frac{1}{2}\sum_{x } |\mu(x)-\nu(x)|=\sum_{x
\in B :\, \mu(x)>\nu(x)}\mu(x)-\nu(x)$.
The worst-case total variation distance at time $t$ is defined as $$d(t) := \max_{x \in \Omega} d_{x}(t), \text{ where for any }x \in \Omega,\,d_{x}(t) :=  \| \Pr_x(X_t \in \cdot)- \pi\|_\mathrm{TV}.$$
The $\epsilon$\textbf{-mixing-time} is defined as  $$t_{\mathrm{mix}}(\epsilon) := \inf \left\{t : d(t) \leq
\epsilon \right\}. $$ 
Similarly, let $d_{\mathrm{ct}}(t):=\max_{x \in \Omega} \|\h_x^t-\pi
\|_\mathrm{TV}$ and let $t_{\mathrm{mix}}^{\mathrm{ct}}(\epsilon):= \inf \left\{t :d_{\mathrm{ct}}(t) \leq
\epsilon \right\}$.

When $\epsilon=1/4$ we omit it from the above
notation.
Next, consider a sequence of such chains, $((\Omega_n,P_n,\pi_n): n \in \N)$, each with its corresponding
worst-distance from stationarity $d^{(n)}(t)$, its mixing-time $t_{\mathrm{mix}}^{(n)}$,
etc.. We say that the sequence exhibits a \emph{\textbf{cutoff}} if the
following
sharp transition in its convergence to stationarity occurs:
$$\lim_{n \to \infty}\frac{t_{\mathrm{mix}}^{(n)}(\epsilon)}{t_{\mathrm{mix}}^{(n)}(1-\epsilon)}=1, \text{ for any }0<\epsilon <1. $$
We say that the sequence has a \emph{\textbf{cutoff window}} $w_n$, if  $w_n=o(t_{\mathrm{mix}}^{(n)})$ and for any $\epsilon \in (0,1)$ there exists $c_{\epsilon}>0$ such that for all $n$ 
\begin{equation}
\label{eq-cutoff-def}
t_{\mathrm{mix}}^{(n)}(\epsilon)-t_{\mathrm{mix}}^{(n)}(1-\epsilon) \le c_{\epsilon} w_n.
\end{equation}
Recall that if $(\Omega,P,\pi)$ is a finite reversible irreducible lazy chain, then $P$ is self-adjoint w.r.t.~the inner product induced by $\pi$ (see Definition \ref{def: L_p distance of measures}) and hence has $|\Omega|$ real eigenvalues. Throughout we shall denote them by $1=\lambda_1>\lambda_2 \ge \ldots \ge \lambda_{|\Omega|} \ge 0$ (where $\lambda_2<1$ since the chain is irreducible and $\lambda_{|\Omega|} \geq 0$ by laziness).
Define the \emph{\textbf{relaxation-time}} of $P$ as $t_{\mathrm{rel}}:=(1-\lambda_2)^{-1}$. The following general relation holds for lazy chains.
\begin{equation}
\label{eq: t_relintro}
(t_{\mathrm{rel}}-1)\log\left(\frac{1}{2\epsilon}\right) \le t_{\mathrm{mix}}(\epsilon) \le
\log \left( \frac{1}{\epsilon \min_x \pi(x)} \right) t_{\mathrm{rel}}
\end{equation}
(see
$\cite{levin2009markov}$ Theorems 12.3 and 12.4). 

We say that a family of chains satisfies the \emph{\textbf{product condition}}  if $(1-\lambda_2^{(n)})t_{\mathrm{mix}}^{(n)} \to \infty$ as $n \to \infty$ (or equivalently, $t_{\mathrm{rel}}^{(n)}=o(t_{\mathrm{mix}}^{(n)})$).
The following well-known fact follows easily from the first inequality
in (\ref{eq: t_relintro}) (c.f.~\cite{levin2009markov}, Proposition 18.4).
\begin{fact}
\label{fact: cutoffandtrel}
For a sequence of irreducible aperiodic reversible Markov chains with relaxation times
$\{t_{\mathrm{rel}}^{(n)} \}$ and mixing-times $\{t_{\mathrm{mix}}^{(n)} \}$, if the sequence exhibits a cutoff, then $t_{\mathrm{rel}}^{(n)}=o(t_{\mathrm{mix}}^{(n)})$.
\end{fact}
In 2004, the third author $\cite{peresamerican}$ conjectured that, in many
natural classes of chains, the product condition is also sufficient for  cutoff.
In general, the product condition does not always imply cutoff. Aldous and Pak (private communication via P.~Diaconis) have constructed
relevant examples (see \cite{levin2009markov}, Chapter 18). This
left open the question of characterizing the classes of chains for which
the product condition is indeed sufficient.

We now state our main theorem, which generalizes previous results concerning birth
and death chains \cite{ding2010total}. The relevant setup is weighted nearest neighbor random walks on finite trees. See Section \ref{sec: trees} for a formal definition.
\begin{maintheorem}
\label{thm: treescutoff}
Let $(V,P,\pi)$ be a lazy reversible Markov chain on a tree $T=(V,E)$ with
$|V| \ge 3$. Then
\begin{equation}
\label{eq: maintreeintro1}
t_{\mathrm{mix}}(\epsilon) - t_{\mathrm{mix}}(1-\epsilon) \le 35\sqrt{\epsilon^{-1}t_{\mathrm{rel}}t_{\mathrm{mix}}}
\text{, for any } 0<\epsilon \le 1/4.
\end{equation}
In particular, if the product condition holds for a sequence of lazy reversible
Markov chains $(V_n,P_n,\pi_n)$ on finite trees $T_n=(V_n,E_n)$, then the
sequence exhibits a cutoff with a cutoff window $w_n=\sqrt{t_{\mathrm{rel}}^{(n)}t_{\mathrm{mix}}^{(n)}}
$.
\end{maintheorem}  

In \cite{diaconis2006separation}, Diaconis and Saloff-Coste showed that a
sequence of birth and death (BD) chains exhibits separation cutoff if and
only if $t_{\mathrm{rel}}^{(n)} = o(t_{\mathrm{mix}}^{(n)})$. In \cite{ding2010total},
Ding et al.~extended this also to the notion of total-variation cutoff and showed that the cutoff window  is always at most $ \sqrt{t_{\mathrm{rel}}^{(n)}t_{\mathrm{mix}}^{(n)}}$
and that in some cases this is tight (see Theorem
1 and Section 2.3 ibid). Since BD chains are a particular case of chains
on trees, the bound on $w_n$ in Theorem \ref{thm:
treescutoff} is also tight. 

We note that the bound we get on the
rate of convergence ((\ref{eq: maintreeintro1})) is better than the estimate in \cite{ding2010total} (even for BD chains), which is $t_{\mathrm{mix}}(\epsilon) - t_{\mathrm{mix}}(1-\epsilon)
\le c\epsilon^{-1} \sqrt{t_{\mathrm{rel}}t_{\mathrm{mix}}} $ (Theorem
2.2). In fact, in Section \ref{sec: refined} we show that under the product condition, $d(t)$ decays in a sub-Gaussian
manner within the cutoff window. More precisely, we show that $t_{\mathrm{mix}}^{(n)}(\epsilon)
- t_{\mathrm{mix}}^{(n)}(1-\epsilon)
\le c \sqrt{t_{\mathrm{rel}}^{(n)}t_{\mathrm{mix}}^{(n)}|\log  \epsilon}| $. This is somewhat similar to 
Theorem 6.1 in \cite{diaconis2006separation}, which determines the ``shape" of the cutoff and describes a necessary
and sufficient spectral condition for the shape to be the density function of the
standard normal distribution.

Concentration
of hitting times was a key ingredient both in \cite{diaconis2006separation}
and \cite{ding2010total} (as it shall be here). Their proofs relied on several
properties which are specific to BD chains. Our proof of Theorem \ref{thm: treescutoff} can be adapted to the following setup.
Denote $[n]:=\{1,2,\ldots,n \}$.
\begin{definition}
\label{def:sbd}
For $n\in \N$ and $\delta, r >0$, we call a finite lazy reversible Markov chain, $([n],P,\pi)$,
a {\bf $(\delta,r)$-semi birth and death (SBD) chain} if 
\begin{itemize}
\item[(i)] For any $i,j \in [n]$ such that
$|i-j|>r$, we have $P(i,j)=0$.
\item[(ii)] For all $i,j \in [n] $
such that $|i-j|=1$, we have that $P(i,j) \ge \delta$. 
\end{itemize}
\end{definition}
This is a natural generalization of the class of birth and death chains.
Conditions (i)-(ii) tie the geometry of the chain to
that of the path $[n]$. We have the following theorem.
\begin{maintheorem}
\label{thm: semibd}
Let $([n_k],P_k,\pi_k)$ be a sequence of $(\delta,r)$-semi birth and death
chains, for some $\delta, r>0$, satisfying the product condition.  Then it exhibits a cutoff with a
cutoff window $w_k:=\sqrt{t_{\mathrm{mix}}^{(k)}t_{\mathrm{rel}}^{(k)}}$.
\end{maintheorem}
We now introduce a new notion of mixing, which shall play a key role in this work.
\begin{definition}
\label{def: worstinprob}
Let $(\Omega,P,\pi)$ be an irreducible chain. For any $x \in 
\Omega$, $\alpha,\epsilon \in (0,1)$ and $t \ge 0$, define $p_{x}(\alpha,t):=\max_{A \subset \Omega:\, \pi(A) \ge \alpha}\Pr_{x}[T_{A}
> t]$, where $T_A:=\inf\{t:X_t \in A \}$ is the \textbf{\textit{hitting time}} of
the set $A$. Set $p(\alpha,t):= \max_{x}p_{x}(\alpha,t)$. We define
$$\mathrm{hit}_{\alpha,x}(\epsilon):=\min
\{t:p_{x}(\alpha,t) \le \epsilon \} \text{ and } \mathrm{hit}_{\alpha}(\epsilon):=\min
\{t:p(\alpha,t)\le \epsilon \}.$$
Similarly, we define $p_{x}^{\mathrm{ct}}(\alpha,t):=\max_{A
\subset \Omega:\, \pi(A) \ge \alpha}\h_{x}[T_{A}^{\mathrm{ct}}
> t] $ (where $T_A^{\mathrm{ct}}:=\inf\{t:X_t^{\mathrm{ct}} \in A \}$) and set $\mathrm{hit}_{\alpha}^{\mathrm{ct}}(\epsilon):=\min
\{t:p_x^{\mathrm{ct}}(\alpha,t)\le \epsilon \text{ for all }x\in \Omega \} $.
\end{definition}

\begin{definition}
\label{def: worstinprobcutoff}
Let $(\Omega_n,P_n,\pi_n)
$ be a sequence of irreducible chains and let $\alpha \in (0,1)$. We say
that the sequence exhibits a $\mathrm{hit}_{\alpha}
$-cutoff, if for any $\epsilon \in (0,1/4)$
$$\mathrm{hit}_{\alpha}^{(n)}(\epsilon)-\mathrm{hit}_{\alpha}^{(n)}(1-\epsilon)
=o \left(\mathrm{hit}_{\alpha}^{(n)}(1/4)
\right).
$$
\end{definition}

We are now ready to state our main abstract theorem. 
\begin{maintheorem}
\label{thm: psigmacutoffequiv}
Let $(\Omega_n,P_n,\pi_n)$ be a sequence of lazy reversible irreducible
finite chains. The following are equivalent:
\begin{itemize}
\item[1)] The sequence exhibits a cutoff.
\item[2)] The
sequence exhibits a $\mathrm{hit}_{\alpha}$-cutoff for some $\alpha \in (0,1/2]$.

\item[3)] The
sequence exhibits a $\mathrm{hit}_{\alpha}$-cutoff for some $\alpha \in (1/2,1)$  and  $t_{\mathrm{rel}}^{(n)}=o(t_{\mathrm{mix}}^{(n)})
$.
\end{itemize}
\end{maintheorem}
\begin{remark}
The proof of Theorem \ref{thm: psigmacutoffequiv} can be extended to the continuous-time case. In particular, it follows that a sequence of finite reversible chains exhibits cutoff iff the sequence of the continuous-time versions of these chains exhibits cutoff. This was previously proven in \cite{chen2013comparison} without the assumption of reversibility.
\end{remark}
\begin{remark}
In Example \ref{ex: Sharpness} we show that there exists a sequence of lazy reversible irreducible
finite Markov chains, $(\Omega_n,P_n,\pi_n)$, such that the product condition fails, yet for all $1/2<\alpha<1$ there is $\hit_{\alpha}$-cutoff. Thus the assertion of Theorem \ref{thm: psigmacutoffequiv} is sharp.
\end{remark}
At first glance $\mathrm{hit}_{\alpha}(\epsilon)$ may seem like a rather weak notion of
mixing compared to $t_{\mathrm{mix}}(\epsilon)$, especially when $\alpha$ is close to 1 (say, $\alpha=1-\epsilon$). The following proposition gives a quantitative version of Theorem \ref{thm: psigmacutoffequiv} (for simplicity we fix $\alpha=1/2$ in (\ref{eq: introTVbound1}) and (\ref{eq: introTVbound2})).
\begin{proposition}
\label{prop: TVbound0}
For any reversible irreducible finite lazy chain and any $\epsilon  \in (0,\frac{1}{4}]$,
\begin{equation}
\label{eq: introTVbound1}
  \mathrm{hit}_{1/2}(3\epsilon/2)-\left\lceil
2t_{\mathrm{rel}}
|\log  \epsilon | \right\rceil \le  t_{\mathrm{mix}}(\epsilon
)\le \mathrm{hit}_{1/2}( \epsilon/2)+\left\lceil t_{\mathrm{rel}} \log
\left(4/ \epsilon \right)\right\rceil \text{ and} 
\end{equation}
\begin{equation}
\label{eq: introTVbound2}
\mathrm{hit}_{1/2}(1-\epsilon/2)-\left\lceil 2t_{\mathrm{rel}}
|\log  \epsilon | \right\rceil   \le t_{\mathrm{mix}}(1-\epsilon) \le \mathrm{hit}_{1/2}(1-2\epsilon)+\left\lceil t_{\mathrm{rel}} \right\rceil.
\end{equation}
Moreover,
\begin{equation}
\label{eq: TVbound0}
\max \{ \mathrm{hit}_{1-\epsilon/4}(5\epsilon/4  ) , (t_{\mathrm{rel}}-1)|\log 2\epsilon
|\} \le t_{\mathrm{mix}}(\epsilon
) \le \mathrm{hit}_{1-\epsilon/4}(3\epsilon/4)+\left\lceil \frac{3t_{\mathrm{rel}}}{2}
\log \left(4/\epsilon \right)\right\rceil.
\end{equation}
Finally, if everywhere in  (\ref{eq: introTVbound1})-(\ref{eq: TVbound0}) $t_{\mathrm{mix}}$ and $\mathrm{hit}$ are replaced by $t_{\mathrm{mix}}^{\mathrm{ct}}$ and $\mathrm{hit}^{\mathrm{ct}}$, respectively, then (\ref{eq: introTVbound1})-(\ref{eq: TVbound0}) still hold (and all ceiling signs can be omitted).
\end{proposition}
\begin{remark}
Define $\rel^{\mathrm{absolute}}:=\max \{(1-\lambda_2)^{-1},(1-|\lambda_{|\Omega |}|)^{-1} \}$. Our only use of the laziness assumption is to argue that $\rel =\rel^{\mathrm{absolute}}$. In particular, Proposition \ref{prop: TVbound0} holds also without the laziness assumption if one replaces $\rel$ by $\rel^{\mathrm{absolute}}$. Similarly, without the laziness assumption the assertion of Theorem \ref{thm: psigmacutoffequiv} should be transformed as follows. A sequence of finite irreducible aperiodic reversible Markov chains exhibits cutoff iff $(\rel^{\mathrm{absolute}})^{(n)}=o(\mix^{(n)}) $ and there exists some $0<\alpha<1$ such that the sequence exhibits $\hit_{\alpha}$-cutoff.

Note that for any finite irreducible reversible chain, $(\Omega,P,\pi)$,
it suffices to consider a $\delta$-lazy version of the chain, $P_{\delta}:=(1-\delta)P+\delta
I$, for some $\delta \ge \frac{1-\max\{\lambda_2
, 0\} }{2}$, to ensure that $\rel
=\rel^{\mathrm{absolute}}$ (which by the previous paragraph, guarantees that
all near-periodicity issues are completely avoided).
\end{remark}  
Loosely speaking, we show that the mixing
of a lazy reversible
Markov chain can be partitioned into two stages as follows. The first is
the time it takes the chain to escape from some small set with sufficiently
large probability.  In the
second stage, the chain mixes at the
fastest possible rate (up to a small constant), which is governed by its
relaxation-time.

It follows from Proposition \ref{prop: submultiplicativityofhit} that the ratio of the LHS and the RHS of (\ref{eq: TVbound0}) is bounded by an absolute constant independent of $\epsilon$. Moreover, (\ref{eq: TVbound0}) bounds $t_{\mathrm{mix}}(\epsilon)$ in terms of hitting distribution of sets of $\pi$ measure tending to 1 as $\epsilon$ tends to 0. In (\ref{eq: hitTV6}) we give a version of (\ref{eq: TVbound0}) for sets of arbitrary $\pi$ measure.  

Either of the two terms appearing in the sum in RHS of (\ref{eq: TVbound0}) may dominate the other. For lazy random walk on two $n$-cliques connected by a single edge, the terms in (\ref{eq: TVbound0}) involving $\mathrm{hit}_{1-\epsilon/4}$ are negligible.
For a sequence of chains satisfying the product condition, all terms in Proposition \ref{prop: TVbound0} involving $t_{\mathrm{rel}}$ are negligible. Hence the assertion of Theorem \ref{thm: psigmacutoffequiv}, for $\alpha=1/2$, follows easily from (\ref{eq: introTVbound1}) and (\ref{eq: introTVbound2}), together with the fact that  $\mathrm{hit}_{1/2}^{(n)}( 1/4)=\Theta(t_{\mathrm{mix}}^{(n)})$. In Proposition \ref{prop: equivoftalphascutoff}, under the assumption that the product condition holds, we prove this fact and show that in fact, if the sequence exhibits $\mathrm{hit}_{\alpha}$-cutoff for some $\alpha \in (0,1)$, then it exhibits $\mathrm{hit}_{\beta}$-cutoff
for all $\beta \in (0,1)$.

\bigskip

An extended abstract of this paper appeared in the proceedings of \emph{ACM-SIAM Symposium on Discrete Algorithms} (SODA), 2015.

\subsection{Related work}
\label{s:related}
The idea that expected hitting times of sets which are ``worst in expectation"
(in the sense of (\ref{eq: tHalpha}) below) could be related to the mixing
time is quite old and goes back to Aldous' 1982 paper \cite{Aldous82}. A
similar result was obtained later by  Lov\'asz and Winkler (\cite{Lovasz98}
Proposition 4.8). 

This aforementioned connection was substantially refined recently by Peres
and Sousi (\cite{peres2011mixing} Theorem 1.1) and independently by Oliveira
(\cite{oliveira2012mixing} Theorem 2). Their approach relied on the theory
of random times to stationarity combined with a certain ``de-randomization"
argument which shows that for any lazy reversible irreducible finite chain
and any stopping time $T$ such that $X_{T} \sim \pi$, $t_{\mathrm{mix}} =
O(\max_{x \in \Omega} \mathbb{E}_x[T])$. As a (somewhat indirect) consequence,
they showed that for any $0<\alpha<1/2$ (this was extended to $\alpha=1/2$
in \cite{griffiths2012tight}), there exist some constants $c_{\alpha},c'_{\alpha}>0$
such that for any lazy reversible irreducible finite chain
\begin{equation}
\label{eq: tHalpha}
c'_\alpha t_{\mathrm{H}}(\alpha) \leq t_{\mathrm{mix}} \leq c_\alpha t_{\mathrm{H}}(\alpha),
\text{ where }t_{\mathrm{H}}(\alpha):=\max_{x \in \Omega }t_{\mathrm{H},x}(\alpha)\text{
and } t_{\mathrm{H},x}(\alpha):= \max_{A \subset \Omega :\,\pi(A)
\ge \alpha}\mathbb{E}_{x}[T_{A}].
\end{equation}

This work was greatly motivated by the aforementioned results. It is natural
to ask whether (\ref{eq: tHalpha}) could be further refined so that the cutoff
phenomenon could be characterized in terms of concentration of the hitting
times of a sequence of sets $A_n \subset \Omega_n$ which attain the maximum in the
definition of $t_{\mathrm{H}}^{(n)}(1/2)$
(starting from the worst initial states). Corollary 1.5 in \cite{Hermon}
asserts that this is indeed the case in the transitive setup. More generally,
Theorem 2 in \cite{Hermon} asserts that this is indeed the case for any fixed
sequence of initial states $x_n \in \Omega_n$ if one replaces $t_{\mathrm{H}}^{(n)}(1/2)$
and $d^{(n)}(t)$ by   $t_{\mathrm{H},x_n}^{(n)}(1/2)$
and $d_{x_{n}}^{(n)}(t)$ (i.e.~when the hitting times and the mixing times
are defined only w.r.t.~these starting states). Alas, Proposition 1.6 in
\cite{Hermon} asserts that in general cutoff could not be characterized in
this manner.

In \cite{lancia2012entropy}, Lancia et al.~established a sufficient condition for cutoff  which does not rely on reversibility. However, their condition includes the strong assumption that for some $A_n \subset \Omega_n$ with $\pi_n(A_n) \ge c > 0$, starting from any $x \in A_n$, the $n$-th chain mixes in $o(t_{\mathrm{mix}}^{(n)})$ steps.
\subsection{An overview of our techniques}
The most important tool we shall utilize is Starr's
$L^2$ maximal inequality (Theorem
\ref{thm: maxergodic}). Relating it to the study of mixing-times of reversible
Markov chains is one of the main contributions of this work.\begin{definition}
\label{def: Goodset}
Let $(\Omega,P,\pi)$ be a finite reversible irreducible lazy chain. Let $A
\subset \Omega$, $s \ge 0$ and $m>0$. Denote $\rho(A):=\sqrt{\Var_{\pi}1_A}=\sqrt{\pi(A)(1-\pi(A))}
$.  Set $\sigma_s:=e^{-s/t_{\mathrm{rel}}}\rho(A)$.  We define
\begin{equation}
\label{eq: GtAm}
G_{s}(A,m):=\left\{ y: |\Pr_y^{k}(A)-\pi(A)| < m\sigma_s \text{
for all }k \ge s \right\}.
\end{equation}
We call the set $G_{s}(A,m)$ the \emph{good set for} $A$ \emph{from time}
$s$ \emph{within }$m$ \emph{standard-deviations}.
\end{definition}
As a simple corollary of Starr's
$L^2$ maximal inequality and the  $L^2$-contraction
lemma we show in Corollary \ref{cor: maxergcor} that for any non-empty $A \subset \Omega$ and any $m,s \ge 0$ that $\pi(G_{s}(A,m))\ge 1-8/m^2$.  To demonstrate the main
idea of our approach we  prove the following inequalities.
\begin{equation}
\label{eq: mainidea}
t_\mathrm{mix}(2\epsilon) \le \mathrm{hit}_{1-\epsilon}(\epsilon)+\left\lceil
\frac{t_{\mathrm{rel}}}{2} \log
\left(\frac{ 2}{\epsilon^3 }\right) \right\rceil. 
\end{equation}
\begin{equation}
\label{eq: mainidea2}
\mathrm{hit}_{1-\epsilon}(1-2\epsilon)\ge t_\mathrm{mix}(1-\epsilon) -\left\lceil
\frac{t_{\mathrm{rel}}}{2} \log
\left(\frac{ 8}{\epsilon^2 }\right) \right\rceil. 
\end{equation}
We first prove (\ref{eq: mainidea}). Fix $A \subset \Omega$ be non-empty. Let $x \in \Omega$.  Let $s,t,m \ge 0$ to be defined
shortly. Denote $G:=G_{s}(A,m)$.
We want this set to be
of size at least $1-\epsilon$. By Corollary \ref{cor: maxergcor} we know
that $\pi(G) \ge 1-8/m^2$. Thus we pick $m=\sqrt{8/\epsilon}$. The precision in (\ref{eq: GtAm}) is  $m\sigma_s
\le \sqrt{8/\epsilon}(\sqrt{\Var_{\pi}1_{A} }e^{-s/t_{\mathrm{rel}}} )\le \sqrt{2/\epsilon}e^{-s/t_{\mathrm{rel}}}$.
We also want
$\epsilon$ precision. Hence we pick $s:=\left\lceil
\frac{t_{\mathrm{rel}}}{2} \log
\left(\frac{ 2}{\epsilon^3 }\right) \right\rceil$. 

We seek to bound $|\Pr_{x}^{t+s}(A)-\pi(A)|$. If  $|\Pr_{x}^{t+s}(A)-\pi(A)|
\le 2\epsilon$, then the chain is ``$2\epsilon$-mixed w.r.t.~$A$". This is
where we use the set $G$. We now demonstrate that for any $t \ge 0$, hitting
$G$ by time $t$ serves as a ``certificate" that the chain is $\epsilon$-mixed w.r.t.~$A$
at time $t+s$. Indeed, from the Markov property and the definition of $G$,
$$|\Pr_{x}[X_{t+s} \in A \mid T_{G} \le t]-\pi(A)| \le \max_{g \in G} \sup_{s'
\ge s}|\Pr_{g}^{s'}(A)-\pi| \le \epsilon.$$
In particular,
\begin{equation}
\label{eq: usingthegoodset}
|\Pr_{x}^{t+s}(A)-\pi(A)| \le \Pr_x[T_{G}>t]+|\Pr_{x}[X_{t+s} \in A \mid
T_{G} \le t]-\pi(A)| \le \Pr_x[T_{G}>t]+\epsilon.
\end{equation}
We seek to have the bound $\Pr_x[T_{G}>t] \le \epsilon$. Recall that by our
choice of $m$ we have that $\pi(G) \ge 1-\epsilon$. Thus if we pick $t:=\mathrm{hit}_{1-\epsilon}(\epsilon)
$, we guarantee that, regardless of the identity of $A$ and $x$, we indeed
have that $\Pr_x[T_{G}>t] \le \epsilon $. Since  $x$ and $A$ were arbitrary, plugging this into (\ref{eq: usingthegoodset})
yields (\ref{eq: mainidea}). We now prove (\ref{eq: mainidea2}).

We now set $r:=t_\mathrm{mix}(1-\epsilon)-1$. Then there exist some $x\in \Omega$ and $A \subset \Omega$ such that $\pi(A)-\Pr_{x}^{r}(A) > 1-\epsilon $. In particular, $\pi(A)>1-\epsilon$. Consider again $G_2:=G_{s_2}(A,m)$. Since again we seek the size of $G_2$ to be at least $1-\epsilon$, we again choose $m=\sqrt{8/\epsilon}$. The precision
in (\ref{eq: GtAm}) is  $m\sigma_{s_{2}}
\le \sqrt{8/\epsilon}(\sqrt{\Var_{\pi}1_{A} }e^{-s_{2}/t_{\mathrm{rel}}} )\le
\sqrt{8/\epsilon}(\sqrt{1-\pi(A)} e^{-s_{2}/t_{\mathrm{rel}}}) \le \sqrt{8}e^{-s_{2}/t_{\mathrm{rel}}}$.
We again seek
$\epsilon$ precision. Hence we pick $s_{2}:=\left\lceil
\frac{t_{\mathrm{rel}}}{2} \log
\left(\frac{ 8}{\epsilon^2 }\right) \right\rceil$. As in (\ref{eq: usingthegoodset}) (with $r-s_2$ in the role of $t$ and $s_2$ in the role of $s$) we have that
$$\Pr_x[T_{G_2}>r-s_2] \ge \pi(A)-\Pr_{x}^{r}(A)-\epsilon > 1-2\epsilon.$$
Hence it must be the case that $\mathrm{hit}_{1-\epsilon}(1-2\epsilon)>r-s_2  = t_\mathrm{mix}(1-\epsilon) -1-\left\lceil
\frac{t_{\mathrm{rel}}}{2} \log
\left(\frac{ 8}{\epsilon^2 }\right) \right\rceil$.
\section{Maximal inequality and applications}
\label{s:trelimplications}
In this section we present the machinery that will be utilized in the proof of the main results. Here and in Section \ref{sec: 3} we only treat the discrete-time chain. The necessary adaptations for the continuous-time case are explained in Section \ref{sec: ct}.  We start with a few basic definitions and facts.

\begin{definition}
\label{def: L_p distance of measures}
Let $(\Omega,P,\pi)$ be a finite reversible chain. For any $f \in \R^{\Omega}$, let $\mathbb{E}_{\pi}[f]:=\sum_{x \in \Omega}\pi(x)f(x)$ and $\Var_{\pi}f:=\mathbb{E}_{\pi}[(f-\mathbb{E}_{\pi}f)^{2}]$. The inner-product $\langle \cdot,\cdot \rangle_{\pi}$ and $L^{p} $ norm are
$$\langle f,g\rangle_{\pi}:=\mathbb{E}_{\pi}[fg] \text{ and } \|f \|_p:=\left( \mathbb{E}_{\pi}[|f|^{p}]\right)^{1/p},\, 1 \le p < \infty$$
We identify the matrix $P^{t}$  with the operator $P^{t}:L^{p}(\R^{\Omega},\pi) \to L^{p}(\R^{\Omega},\pi)$ defined by $P^{t}f(x):=\sum_{y \in \Omega}P^{t}(x,y)f(y)=\mathbb{E}_{x}[f(X_{t})]$. Then by reversibility $P^{t}:L^2 \to L^2$ is a self-adjoint operator. 
\end{definition}
The spectral decomposition in discrete time takes the following form.
If $f_1,\ldots,f_{|\Omega|}$ is an orthonormal basis of $L^{2}(\R^{\Omega},\pi)$ such that $Pf_i:=\lambda_i
f_i$ for all $i$, then $P^t g=\mathbb{E}_{\pi}P^t g+ \sum_{i=2}^{|\Omega|}
\langle g,f_{i}\rangle_{\pi}\lambda_i^tf_i$, for all
$g \in \R^{\Omega}$ and $t \ge 0$. The following lemma is standard. It is proved using the spectral decomposition in a straightforward manner.\begin{lemma}[$L^2$-contraction Lemma]
\label{lem: L2exp}
Let $(\Omega,P,\pi)$ be a finite lazy reversible irreducible Markov chain. Let $f \in \R^{\Omega}$.  Then

\begin{equation}
\label{eq: L2contraction0}
\Var_{\pi}P^tf \le e^{-2t/t_{\mathrm{rel}}}\Var_{\pi}f, \text{ for any }t \ge 0.
\end{equation}
\end{lemma}

We now state a particular case of Starr's maximal inequality (\cite{starr1966operator} Theorem 1). It is similar to Stein's maximal inequality
(\cite{stein1961maximal}), but gives the best possible constant. For the sake of completeness we also prove Theorem \ref{thm: maxergodic} at the end of this section.
\begin{theorem}[Maximal inequality]
\label{thm: maxergodic}
Let $(\Omega,P,\pi)$ be a reversible irreducible Markov chain. Let $1<p<\infty$. Then for any $f \in L^{p}(\R^{\Omega},\pi)
$,
\begin{equation}
\label{eq: ergodic1}
\|f^{*} \|_{p} \le \left( \frac{p}{p-1} \right) \|f \|_p,
\end{equation}
where $f^* \in \R^\Omega$ is the corresponding \textbf{\textit{maximal function at even times}}, defined as $$f^{*}(x):=\sup_{0 \le k < \infty}|P^{2k}(f)(x)|=\sup_{0 \le k < \infty}|\mathbb{E}_{x}[f(X_{2k})]|.$$
\end{theorem}
The following corollary follows by combining Lemma \ref{lem: L2exp} with Theorem \ref{thm: maxergodic}. 
\begin{corollary}
\label{cor: maxergcor}
Let $(\Omega,P,\pi)$ be a finite reversible irreducible lazy chain. As in Definition \ref{def: Goodset}, define $\rho(A):=\sqrt{\pi(A)(1-\pi(A))}$,  $\sigma_t:=\rho(A)e^{-t/t_{\mathrm{rel}}}$ and
$$G_{t}(A,m):=\left\{ y: |\Pr_y^{k}(A)-\pi(A)| < m\sigma_t \text{
for all }k \ge t \right\}.$$
Then 
\begin{equation}
\label{eq: G_t}
\pi(G_{t}(A,m)) \ge 1-8m^{-2}, \text{ for all }A \subset \Omega,\, t\ge 0 \text{ and }m>0.
\end{equation}
\end{corollary}

\begin{proof}
For any $t \ge0$, let $f_t(x):=P^{t}(1_{A}(x)-\pi(A)) =\Pr_x^t(A)-\pi(A)$.
Then in the notation of Theorem \ref{thm: maxergodic}, $$f_t^{*}(x):=\sup_{k
\ge 0}|P^{2k}f_t(x) |=\sup_{k \ge 0}|\Pr_x^{2k+t}(A)-\pi(A) |,$$ and similarly
$$(Pf_{t})^{*}(x)=\sup_{k
\ge 0}|\Pr_x^{2k+1+t}(A)-\pi(A) |.$$ Hence $G_{t}=\left\{x \in \Omega:f_t^*(x),(Pf_t)^*(x)
<  m\sigma_{t}\right\}$.  Whence
\begin{equation}
\label{eq: ergodic3}
\begin{split}
& 1- \pi(G_{t}) \le \pi \left\{x:f_{t}^* (x) \ge  m\sigma_{t}\right\}+\pi \left\{x:(Pf_{t})^*(x)
\ge  m\sigma_{t}\right\}.
 \end{split}
\end{equation}
Note that since $\pi P^{t}=\pi$ we have that $\mathbb{E}_{\pi}(f_t)=\mathbb{E}_{\pi}(f_0)=\mathbb{E}_{\pi}(1_{A}-\pi(A))=0$.
Now (\ref{eq: L2contraction0}) implies
that
\begin{equation}
\label{eq: ergodic4}
\|Pf_t \|_2^2 \le \|f_t \|_2^{2}=\Var_{\pi}P^tf_0 \le e^{-2t/t_{\mathrm{rel}}}\Var_{\pi}
f_0=e^{-2t/t_{\mathrm{rel}}}\rho^2(A)=\sigma_t^{2}.
\end{equation}
Hence by Markov inequality and (\ref{eq: ergodic1}) we have
\begin{equation}
\label{eq: ergodic6}
\pi \left\{x:f_t^{*}(x)\ge
m\sigma_{t} \right\}=\pi \left\{x:(f_t^{*}(x))^{2}\ge
m^{2}\sigma_{t}^2 \right\}  \le 4m^{-2},
\end{equation}
and similarly, $\pi \left\{x:(Pf_{t})^*(x)
\ge  m\sigma_{t}\right\} \le 4 m^{-2}$.

The corollary now follows by substituting the last two bounds in (\ref{eq:
ergodic3}).
\end{proof}

\subsection{Proof of Theorem \ref{thm: maxergodic}}
As promised, we end this section with the proof of Theorem \ref{thm: maxergodic}.

\emph{Proof of Theorem \ref{thm: maxergodic}.} Let $p \in (1,\infty)$ and $f \in L^p(\R^{\Omega},\pi)$. Let $q:= \frac{p}{p-1}$ be the conjugate exponent of $p$. We argue that it suffices to prove the theorem only for $f \ge 0$, since for general $f$, if we denote $h:=|f|$, then $|f_{*}| \le h_{*}$. Consequently, $\|f_{*}\|_p \le \|h_{*}\|_p \le q \|h\|_p=q\| f \|_p $.

Let $(X_{n})_{n \ge 0 }$ have the distribution of the chain $(\Omega,P,\pi)$ with $X_{0} \sim \pi$. Let $n \ge 0$. Let $0 \le f \in L^p(\Omega,\pi)$. By the tower property of conditional expectation (e.g.~\cite{durrett2010probability},
Theorem 5.1.6.),
\begin{equation}
\label{eq: starr1}
 P^{2n}f(X_0):=  \mathbb{E}[f(X_{2n})\mid X_0]=\mathbb{E}[ \mathbb{E}[f(X_{2n})\mid X_n]
\mid X_0]=
\mathbb{E}[R_{n}\mid X_0],
\end{equation}
where $R_n:= \mathbb{E}[f(X_{2n})\mid X_n]$.
 Since $X_0 \sim \pi$, by reversibility,  $(X_n,X_{n+1},\ldots,X_{2n})$ and $(X_n,X_{n-1},\ldots,X_0)$ have the same law.  Hence 
\begin{equation}
\label{eq: starr2}
R_n=\mathbb{E}[f(X_{2n}) \mid X_n]=\mathbb{E}[f(X_{0}) \mid X_n]=\mathbb{E}[f(X_{0}) \mid X_n,X_{n+1},\ldots],
\end{equation}
where the second equality in (\ref{eq: starr2}) follows by the Markov property. Fix $N \ge 0$.  By (\ref{eq: starr2}) $(R_n)_{n = 0}^{N}$ is a reverse martingale, i.e.~$(R_{N-n})_{n=0}^{N}$ is a martingale. By Doob's $L^p$ maximal inequality (e.g.~\cite{durrett2010probability}, Theorem 5.4.3.)
\begin{equation}
\label{eq: Doob}
\|\max_{0 \le n \le N}R_n \|_p \le q\|R_0\|_p=q\|f(X_0)\|_p.
\end{equation}
Denote $h_N:=\max_{0 \le n \le N}P^{2n}f$. By (\ref{eq: starr1}),
\begin{equation}
\label{eq: Starr5}
h_{N}(X_{0})= \max_{0 \le n \le N}\mathbb{E}[R_{n}\mid X_0]  \le \mathbb{E}\left[ \max_{0 \le n \le N}R_{n}\mid X_0\right].
\end{equation}
By conditional Jensen inequality $\|\mathbb{E}[Y \mid X_0] \|_p \le \|Y\|_p $ (e.g.~\cite{durrett2010probability}, Theorem 5.1.4.). So by taking $L^p$ norms in (\ref{eq: Starr5}), together with (\ref{eq: Doob}) we get that
\begin{equation}
\label{eq: starr3}
\begin{split}
& \|h_N \|_p \le \|\max_{0 \le n \le N}R_n \|_p \le q\|f(X_0)\|_p.
\end{split}
\end{equation}
The proof is concluded using the monotone convergence theorem. \qed
\vspace{2mm}

\section{Inequalities relating $t_{\rm mix}(\epsilon)$ and ${\rm hit_{\alpha}(\delta)}$}
\label{sec: 3}
Our aim in this section is to obtain inequalities relating $t_{\rm mix}(\epsilon)$ and ${\rm hit_{\alpha}(\delta)}$ for suitable values of $\alpha$, $\epsilon$ and $\delta$ using Corollary \ref{cor: maxergcor}.

The following corollary uses the same reasoning as in the proof of (\ref{eq: mainidea})-(\ref{eq: mainidea2}) with a slightly more careful analysis.
\begin{corollary}
\label{cor: hitprob}
Let $(\Omega,P,\pi)$ be a lazy reversible irreducible finite chain.
Let $x \in \Omega$,  $\delta,\alpha \in (0,1)$,  $s\ge 0$
and  $A \subset \Omega$. Denote $t:=\mathrm{hit}_{1-\alpha,x}(\delta)$.
Then
\begin{equation}
\label{eq: hitprob1}
\Pr_{x}^{t+s}[A] \ge (1-\delta)\left[\pi(A)-e^{-s/t_{\mathrm{rel}}}\left[8 \alpha^{-1}\pi(A)(1-\pi(A))
\right]^{1/2} \right].
\end{equation}
Consequently, for any $0< \epsilon
<1$ we have that 
\begin{equation}
\label{eq: hitTV6}
\mathrm{hit}_{1-\alpha}((\alpha+\epsilon) \wedge 1) \le t_{\mathrm{mix}}(\epsilon)
\text{ and } t_{\mathrm{mix}}((\epsilon+\delta)\wedge 1)\le \mathrm{hit}_{1-\alpha}(\epsilon)+\left\lceil  \frac{t_{\mathrm{rel}}}{2}\log^{+}
\left(\frac{2(1-\epsilon)^2}{\alpha \epsilon \delta} \right)\right\rceil,  
\end{equation}
where $a\wedge b:=\min \{a,b\}$ and $\log^{+}x:=\max \{ \log x,0\}$. In particular,
for any $0< \epsilon \le 1/2$,  
\begin{equation}
\label{eq: hitTV1}
\mathrm{hit}_{1-\epsilon/4}(5\epsilon/4  ) \le t_{\mathrm{mix}}(\epsilon ) \le \mathrm{hit}_{1-\epsilon/4}(3\epsilon/4)+\left\lceil \frac{3t_{\mathrm{rel}}}{2} \log
\left(4/\epsilon \right)\right\rceil,
\end{equation}
\begin{equation}
\label{eq: hitTV2}
t_{\mathrm{mix}}(\epsilon
)\le \mathrm{hit}_{1/2}( \epsilon/2)+\left\lceil t_{\mathrm{rel}} \log
\left(4/ \epsilon \right)\right\rceil \text{ and } t_{\mathrm{mix}}(1-\epsilon/2 ) \le   \mathrm{hit}_{1/2}(1-\epsilon)+\left\lceil t_{\mathrm{rel}}  \right\rceil. 
\end{equation}
\end{corollary}


\begin{proof}
We first prove (\ref{eq: hitprob1}). Fix some $x \in \Omega$. Consider the set $$G=G_{s}(A):=\left\{ y: |\Pr_y^{k}(A)-\pi(A)| < e^{-s/t_{\mathrm{rel}}}\left(8
\alpha^{-1}\pi(A)(1-\pi(A))
\right)^{1/2} \text{
for all }k \ge s \right\}.$$
Then by Corollary \ref{cor: maxergcor} we have that $$\pi(G)\ge 1-\alpha.
$$ 
By the Markov property and
conditioning on $T_{G}$ and on $X_{T_{G}}$ we get that
$$\Pr_{x}^{t+s}[A \mid T_{G} \le t] \ge \pi(A)-e^{-s/t_{\mathrm{rel}}} \left[8
\alpha^{-1}\pi(A)(1-\pi(A))
\right]^{1/2}.$$
Since $\pi(G)  \ge 1-\alpha$ we have that $\Pr_{x}[ T_{G} \le t] \ge
1-\delta$ for $t:=\mathrm{hit}_{1-\alpha,x}(\delta) $. Thus
$$\Pr_{x}^{t+s}[A] \ge \Pr_{x}[T_{G} \le t]\Pr_{x}^{t+s}[A
\mid T_{G} \le t]\ge (1-\delta)\left[ \pi(A)-e^{-s/t_{\mathrm{rel}}} \left[8
\alpha^{-1}\pi(A)(1-\pi(A))
\right]^{1/2}\right],$$
which concludes the proof of (\ref{eq: hitprob1}). We now prove (\ref{eq: hitTV6}). The first inequality in (\ref{eq: hitTV6}) follows directly from the definition
of the total variation distance. To see this, let $A \subset \Omega$ be an arbitrary set with $\pi(A) \ge 1-\alpha$.
Let $t_{1}:=t_{\mathrm{mix}}(\epsilon)$. Then for any $x \in \Omega$, $\Pr_{x}[T_{A} \le t_{1}] \ge \Pr_{x}[
X_{t_1} \in A] \ge \pi(A)-\|\Pr_{x}^{t_{1}} - \pi \|_{\mathrm{TV}}  \ge 1-\alpha-\epsilon$. In particular, we get directly from Definition \ref{def: worstinprobcutoff} that $\mathrm{hit}_{1-\alpha}(\alpha+\epsilon) \le t_1=t_{\mathrm{mix}}(\epsilon)$. We now prove the second inequality in (\ref{eq: hitTV6}).

Set $t:=\mathrm{hit}_{1-\alpha}(\epsilon)$ and $s:=\left\lceil \frac{1}{2}t_{\mathrm{rel}}\log^
{+} \left(\frac{2(1-\epsilon)^2}{  
\alpha \epsilon \delta } \right)\right\rceil$. Let $x \in \Omega$ be such that $d(t+s,x)=d(t+s) $ and set $A:=\{y \in \Omega: \pi(y)> \Pr_{x}^{t+s}(y) \}$. Observe that by the choice of $t,s,x$ and $A$ together with (\ref{eq: hitprob1}) we have that
\begin{equation}
\label{eq: hitTV4}
\begin{split}
& d(t+s)=\pi(A)-\Pr_{x}^{t+s}(A) \le \epsilon \pi(A)+(1-\epsilon)e^{-s/t_{\mathrm{rel}}}\left[8
\alpha^{-1}\pi(A)(1-\pi(A))
\right]^{1/2} \\ & \le \epsilon[\pi(A) +2\sqrt{\delta/\epsilon}\sqrt{\pi(A)(1-\pi(A))} ] \le \epsilon[1+(2\sqrt{\delta/\epsilon})^2/4]=\epsilon+\delta,
\end{split}
\end{equation}
where in the last inequality we have used the easy fact that for any $c >0$ and any $x \in [0,1]$ we have that $x+c\sqrt{x(1-x)}\le 1+c^2/4$. Indeed, since $x \in [0,1]$ it suffices to show that $x+c\sqrt{(1-x)}\le 1+c^2/4
$. Write $\sqrt{1-x}=y$ and $c/2=a$. By subtracting $x$ from both sides,
the previous inequality is equivalent to $2ay \le y^2+a^2$. This concludes the proof of (\ref{eq: hitTV6}).

To get (\ref{eq: hitTV1}), apply (\ref{eq: hitTV6}) with $(\alpha,\epsilon,\delta)$ being $(\epsilon/4,3\epsilon/4,\epsilon/4)$.
Similarly, to get (\ref{eq: hitTV2}) apply (\ref{eq: hitTV6}) with $(\alpha,\epsilon,\delta )$ being $(1/2,\epsilon/2,\epsilon/2)$ or $(1/2,1-\epsilon,\epsilon/2 )$, respectively. 
\end{proof}
\begin{remark}
Corollary \ref{cor: hitprob} holds also in continuous-time
case (where everywhere in (\ref{eq: hitprob1})-(\ref{eq: hitTV2})
$t_{\mathrm{mix}}$ and $\mathrm{hit}$ are replaced by $t_{\mathrm{mix}}^{\mathrm{ct}}$
and $\mathrm{hit}^{\mathrm{ct}}$, respectively,  and all ceiling signs are
omitted). The necessary adaptations are explained in Section \ref{sec: ct}.
\end{remark}
Let $\alpha \in (0,1)$. Observe that for any $A \subset \Omega$ with $\pi(A) \ge \alpha$, any $x \in \Omega$ and any $t,s \ge 0$ we have that $\Pr_{x}[T_{A} > t+s] \le \Pr_{x}[T_{A} > t]\bigl( \max_{z}\Pr_{z}[T_{A} > s]\bigr) \le p(\alpha,t)p(\alpha,s) $. Maximizing over $x$ and $A$ yields that $p(\alpha,t+s) \le p(\alpha,t)p(\alpha,s)$, from which the following proposition follows. 
\begin{proposition}
\label{prop: submultiplicativityofhit}
For any $\alpha,\epsilon,\delta \in (0,1)$ we have that
\begin{equation}
\label{eq: submult}
\mathrm{hit}_{\alpha}(\epsilon \delta) \le \mathrm{hit}_{\alpha}(\epsilon )+\mathrm{hit}_{\alpha}(\delta).
\end{equation}
\end{proposition}

\vspace{2mm}

In the next corollary, we establish inequalities between ${\rm hit}_{\alpha}(\delta)$ and ${\rm hit}_{\beta}(\delta')$ for appropriate values of $\alpha, \beta, \delta$ and $\delta'$.

\begin{corollary}
\label{prop: hitpqinequalities}
For any reversible irreducible finite chain and $0< \epsilon < \delta  < 1$,
\begin{equation}
\label{eq: thitinequality1}
\mathrm{hit}_{\beta }(\delta) \le \mathrm{hit}_{\alpha }(\delta)
\le \mathrm{hit}_{\beta }(\delta- \epsilon)+ \left\lceil \alpha^{-1} t_{\mathrm{rel}} \log \left( \frac{1-\alpha}{(1-\beta)\epsilon} \right) \right\rceil, \text{ for any } 0 < \alpha \le \beta < 1.
\end{equation}
\end{corollary}

The general idea behind Corollary \ref{prop: hitpqinequalities} is as follows. Loosely speaking, we show that any (not too small) set  $A \subset
\Omega$ has a ``blow-up" set $H(A)$ (of large $\pi$-measure), such that starting from any $x \in H(A)$, the set $A$ is hit ``quickly" (in time proportional to $t_{\mathrm{ rel}}$ times a constant depending on the size of $A$) with large probability.

In order to establish the existence of such a blow-up, it turns out that it suffices to consider the hitting time of $A$, starting from the initial distribution $\pi$, which is well-understood.\begin{lemma}
\label{lem: AF1}
Let $(\Omega,P,\pi)$ be a finite irreducible reversible Markov chain. Let
$A \subsetneq\ \Omega$ be non-empty. Let $\alpha>0$ and $w \ge 0$. Let $B(A,w,\alpha):=\left\{y:\Pr_{y}
\left[T_{A} > \left\lceil \frac{t_{\mathrm{rel}}w }{\pi(A)}\right\rceil
\right] \ge \alpha  \right\}$. Then
\begin{equation}
\label{eq: CM1}
\Pr_{\pi}[T_{A} > t]  \le \pi(A^c)   
\left(
1-\frac{\pi(A)}{t_{\mathrm{rel}}} \right)^{t}  \le \pi(A^c)  \exp \left(- \frac{t
\pi(A)}{t_{\mathrm{rel}}} \right), \text{ for any }t \ge 0.
\end{equation}
In particular, 
\begin{equation}
\label{eq: CM2}
\pi \left(B(A,w,\alpha) \right) \le \pi(A^c)   e^{-w}\alpha^{-1} \text{ and }\pi(A)\mathbb{E}_{\pi}[T_{A}]
 \le t_{\mathrm{rel}}\pi(A^c).
\end{equation}
\end{lemma}
The proof of Lemma
\ref{lem: AF1} is deferred to the end of this section.

\emph{Proof of Corollary \ref{prop: hitpqinequalities}.} Denote $s=s_{\alpha,\beta,\epsilon}:= \left\lceil \alpha^{-1}
t_{\mathrm{rel}} \log \left( \frac{1-\alpha}{(1-\beta)\epsilon} \right) \right\rceil $. Let $A \subset \Omega$ be an arbitrary set such that $\pi(A) \ge \alpha $. Consider the
set $$H_1=H_1(A,\alpha ,\beta, \epsilon ):=\left\{y \in \Omega:\Pr_{y}[T_{A} \le s ] \ge 1-\epsilon
  \right\}.$$ Then by (\ref{eq:
CM2}) \begin{equation*}
\begin{split}
\pi(H_1) & \ge 1-(1-(1-\epsilon) )^{-1} (1-\pi(A))\exp \left[-\frac{s\pi(A) }{ t_{\mathrm{rel}}} \right]
 \\ & \ge 1-\epsilon^{-1} (1-\alpha)
 \exp \left[- \log \left(\frac{1-\alpha}{(1-\beta)\epsilon}\right)
\right]=\beta .
\end{split}
\end{equation*}
 By the definition of $H_1$ together with  the Markov property and the fact
that $\pi(H_{1}) \ge \beta  $, for any $t \ge 0$ and $x \in \Omega$,
\begin{equation}
\label{eq: concentrationallsizes1}
\begin{split}
\Pr_{x}[T_{A}
& \le t+s] \ge  \Pr_{x}[T_{H_1} \le t, T_{A}
\le t+s]  \ge (1-\epsilon) \Pr_{x}[T_{H_1} \le t] \\ & \ge
(1-\epsilon)(1- p_{x}(\beta ,t)) \ge 1-\epsilon - \max_{y \in \Omega}p_{y}(\beta ,t) .
\end{split}
\end{equation}
Taking $t:=\mathrm{hit}_{\beta }(\delta -\epsilon )$ and minimizing the LHS of (\ref{eq: concentrationallsizes1}) over $A$ and $x$ gives the second inequality in (\ref{eq: thitinequality1}). The first inequality in (\ref{eq: thitinequality1})  is trivial because $\alpha \le
\beta $. \qed
\vspace{2mm}

\subsection{Proofs of Proposition \ref{prop: TVbound0} and Theorem \ref{thm: psigmacutoffequiv}}
Now we are ready to prove our main abstract results.

\begin{proof}[Proof of Proposition \ref{prop: TVbound0}] First note that (\ref{eq:
TVbound0}) follows from (\ref{eq: hitTV1}) and the first inequality in (\ref{eq:
t_relintro}). Moreover, in light of (\ref{eq: hitTV2}) we only need to prove
the first inequalities in (\ref{eq: introTVbound1}) and (\ref{eq: introTVbound2}). Fix some $0<\epsilon\le 1/4$. Take any set $A$ with $\pi(A)\geq \frac{1}{2}$ and $x\in \Omega$. Denote $s_{\epsilon}:=\lceil 2t_{\rm
rel}|\log \epsilon | \rceil$ It follows by coupling the chain with initial distribution $\Pr_x^t$ with the stationary chain that for all $t\geq 0$  
$$\Pr_x[T_A>t+s_{\epsilon}]\leq d_{x}(t)+\Pr_{\pi}[T_A>s_{\epsilon}]\leq d_{x}(t)+\frac{1}{2}e^{-s_{\epsilon}/2t_{\rm rel}}\leq d(t)+\frac{\epsilon}{2}$$
where the penultimate inequality above is a consequence of (\ref{eq: CM1}). Putting $t=t_{\rm mix}(\epsilon)$ and $t=t_{\rm mix}(1-\epsilon)$ successively in the above equation and maximizing over $x\in \Omega$ and $A$ such that $\pi(A)\geq \frac{1}{2}$ gives
$${\rm hit}_{1/2}(3\epsilon/2)\leq t_{\rm mix}(\epsilon)+ s_{\epsilon}~\text{and}~{\rm hit}_{1/2}(1-\epsilon/2)\leq t_{\rm mix}(1-\epsilon)+s_{\epsilon},$$
which completes the proof.     
\end{proof}

Before completing the proof of Theorem \ref{thm: psigmacutoffequiv}, we prove that under the product condition if a sequence of reversible chains exhibits ${\rm hit}_{\alpha}$-cutoff for some $\alpha\in (0,1)$, then it exhibits ${\rm hit}_{\alpha}$-cutoff for all $\alpha\in (0,1)$.



\begin{proposition}
\label{prop: equivoftalphascutoff}
Let $(\Omega_n,P_n,\pi_n)$ be a sequence of lazy finite irreducible reversible
chains. Assume that the product condition holds. Then (1) and (2) below are
equivalent:
\begin{itemize}
\item[(1)] There exists $\alpha \in (0,1)$ for which the
sequence exhibits a $\mathrm{hit}_{\alpha}$-cutoff.
\item[(2)] The sequence
exhibits a $\mathrm{hit}_{\alpha}$-cutoff for any $\alpha
\in (0,1)$.
\end{itemize}
Moreover,
\begin{equation}
\label{eq: hitThetamix}
\mathrm{hit}_{\alpha}^{(n)}(1/4)
= \Theta (t_{\mathrm{mix}}^{(n)}), \text{ for any } \alpha \in (0,1).
\end{equation}
Furthermore, if (2) holds then 
\begin{equation}
\label{eq: ratiohit}
\lim_{n \to \infty}
\mathrm{hit}_{\alpha}^{(n)}(1/4)/\mathrm{hit}_{1/2}^{(n)}(1/4)
= 1, \text{ for any } \alpha \in (0,1).
\end{equation}
\end{proposition}

\begin{proof} 
We start by proving (\ref{eq: hitThetamix}). Assume that the product condition holds. Fix some $\alpha \in (0,1)$. Note that we have
$$\mathrm{hit}_{\alpha}^{(n)}(1/4) \le 4\alpha^{-1} \mathrm{hit}_{\alpha}^{(n)}\left (1-\frac{3\alpha}{4}\right) \le 4\alpha^{-1} t_{\mathrm{mix}}^{(n)}\left(\frac{\alpha}{4}\right) \le 4\alpha^{-1}  (2+\lceil \log_{2} (1/\alpha)\rceil ) t_{\mathrm{mix}}^{(n)}.$$
The first inequality above follows from (\ref{eq: submult}) and the fact that $(1-3\alpha/4)^{4\alpha^{-1}-1} \le 4 e^{-3} \le 1/4$. The second one follows from (\ref{eq:
hitTV6})(first inequality). The final inequality above is a consequence of the sub-multiplicativity property: for any $k,t \ge 0$,  $d(kt) \le (2d(t))^k$ (e.g.~\cite{levin2009markov}, (4.24) and Lemma 4.12).

Conversely, by (\ref{eq: submult}) (second inequality) and the second inequality in (\ref{eq:
hitTV6}) with $(\alpha,\epsilon,\delta)$ here being $(1-\alpha,1/8,1/8) $ (first inequality)
$$ \frac{t_{\mathrm{mix}}^{(n)}}{2} - \left\lceil \frac{t_{\mathrm{rel}}^{(n)}}{4}\log \left(\frac{100}{1-\alpha}\right) \right\rceil \le \frac{\mathrm{hit}_{\alpha}^{(n)}(1/8)}{2} \le \mathrm{hit}_{\alpha}^{(n)}(1/4).$$
This concludes the proof of (\ref{eq: hitThetamix}).
We now prove the equivalence between (1) and (2) under the product condition. It suffices
to show that (1) $\Longrightarrow$ (2), as the reversed implication is trivial. Fix $0<\alpha < \beta < 1$. It suffices to show that $\mathrm{hit}_{\alpha}$-cutoff occurs iff $\mathrm{hit}_{\beta}$-cutoff occurs.

Fix $\epsilon \in (0,1/8)$. Denote $s_{n}=s_{n}(\alpha,\beta,\epsilon):=
\left\lceil t_{\mathrm{rel}}^{(n)} \alpha^{-1}  \log \left(\frac{1-\alpha}{(1-\beta)\epsilon} \right) \right\rceil$. By the second inequality in Corollary
\ref{prop: hitpqinequalities}    
\begin{equation}
\label{eq: concentrationallsizes5'}
\begin{split}
\mathrm{hit}_{\alpha}^{(n)}(1-\epsilon)
 \le \mathrm{hit}_{\beta}^{(n)}(1-2\epsilon)+ s_{n} \text{ and } \mathrm{hit}_{\alpha}^{(n)}(2\epsilon)  \le
\mathrm{hit}_{\beta}^{(n)}(\epsilon)+
s_{n}.
\end{split}
\end{equation}
By the first
inequality in Corollary
\ref{prop: hitpqinequalities}
\begin{equation}
\label{eq: concentrationallsizes5}
\mathrm{hit}_{\beta}^{(n)}(2\epsilon) \le\mathrm{hit}_{\alpha}^{(n)}(2\epsilon)\le \mathrm{hit}_{\alpha}^{(n)}(\epsilon)
\text{ and } \mathrm{hit}_{\beta}^{(n)}(1-\epsilon) \le \mathrm{hit}_{\beta}^{(n)}(1-2\epsilon)
\le \mathrm{hit}_{\alpha}^{(n)}(1-2\epsilon).
\end{equation}
Hence
\begin{equation}
\label{eq: hitcutoffequivdifsize7}
\begin{split}
\mathrm{hit}_{\beta}^{(n)}(2\epsilon)-\mathrm{hit}_{\beta}^{(n)}(1-2\epsilon)
& \le \mathrm{hit}_{\alpha}^{(n)}(\epsilon)-\mathrm{hit}_{\alpha}^{(n)}(1-\epsilon)+
s_{n},
\\ \mathrm{hit}_{\alpha}^{(n)}(2\epsilon)-\mathrm{hit}_{\alpha}^{(n)}(1-2\epsilon)
& \le \mathrm{hit}_{\beta}^{(n)}(\epsilon)-\mathrm{hit}_{\beta}^{(n)}(1-\epsilon)+s_{n}.
\end{split}
\end{equation}
Note that by the assumption that the product condition holds, we have that $s_n=o(t_{\mathrm{mix}}^{(n)})$. Assume that the sequence exhibits $\mathrm{hit}_{\alpha}$-cutoff. Then by (\ref{eq: hitThetamix}) the RHS of the first line of (\ref{eq: hitcutoffequivdifsize7}) is $o(t_{\mathrm{mix}}^{(n)})$. Again by (\ref{eq: hitThetamix}), this implies that the RHS of the first line of (\ref{eq: hitcutoffequivdifsize7})
is $o(\mathrm{hit}_{\beta}^{(n)}(1/4))$ and so the sequence exhibits $\mathrm{hit}_{\beta}$-cutoff. Applying the same reasoning, using the second line of (\ref{eq: hitcutoffequivdifsize7}), shows that if the sequence exhibits $\mathrm{hit}_{\beta}$-cutoff, then it also exhibits $\mathrm{hit}_{\alpha}$-cutoff.

We now prove (\ref{eq: ratiohit}). Let $a \in (0,1)$. Denote $\alpha:=\min \{a,1/2 \}$ and $\beta:=\max \{a,1/2 \}$. Let $s_n=s_{n}(\alpha,\beta,\epsilon)$ be as before. 
By the second inequality in Corollary
\ref{prop: hitpqinequalities} 
\begin{equation}
\label{eq: concentrationallsizes6'}
   \mathrm{hit}_{\alpha}^{(n)}(1/4+\epsilon)-s_{n} \le \mathrm{hit}_{\beta}^{(n)}(1/4 
) \le \mathrm{hit}_{\alpha}^{(n)}(1/4 
).
\end{equation}
By assumption (2) together with the product condition and (\ref{eq: hitThetamix}), the LHS of (\ref{eq: concentrationallsizes6'}) is at least $(1-o(1)) \mathrm{hit}_{\alpha}^{(n)}(1/4) $, which by (\ref{eq: concentrationallsizes6'}), implies (\ref{eq: ratiohit}).
\end{proof}

The following proposition shows that the product condition is implied by $\mathrm{hit}_{\alpha}$-cutoff for any $\alpha \leq 1/2$. In particular, this implies the equivalence of $2)$ and $3)$ in Theorem \ref{thm: psigmacutoffequiv}.

\begin{proposition}
\label{prop: prodcondandhitcutoff}
Let $(\Omega_n,P_n,\pi_n)$ be a sequence of lazy finite irreducible reversible
chains. Assume  that  the product condition fails. Then for any $\alpha \le 1/2 $ the sequence does not exhibit $\mathrm{hit}_{\alpha}$-cutoff.
\end{proposition}

Before providing the proof of Proposition \ref{prop: prodcondandhitcutoff}, we complete the proof of Theorem \ref{thm: psigmacutoffequiv}.  

\begin{proof}[Proof of Theorem \ref{thm: psigmacutoffequiv}] By Fact \ref{fact:
cutoffandtrel} and Proposition \ref{prop: prodcondandhitcutoff} it suffices
to consider the case in which the product condition holds. By Propositions
\ref{prop: equivoftalphascutoff}   it suffices to consider the case $\alpha=1/2$
(that is, it suffices to show that under the product condition the sequence
exhibits cutoff iff it exhibits $\mathrm{hit}_{1/2}$-cutoff). This follows
at once from (\ref{eq: introTVbound1}), (\ref{eq: introTVbound2}) and (\ref{eq:
hitThetamix}). 
\end{proof}

\begin{proof}[Proof of Proposition \ref{prop: prodcondandhitcutoff}]
Fix some $0< \alpha \le 1/2 $. We first argue that for all $n$, $k\ge 1$

\begin{equation}
\label{eq: nohitcutoff2}
\hit_{\alpha}^{(n)}([1-\alpha/2]^{k} ) \le k  \lceil| \log_2 (\alpha /2)|
\rceil   \mix^{(n)}.
\end{equation}
By the submultiplicativity property (\ref{eq: submult}), it suffices to verify
(\ref{eq: nohitcutoff2}) only for $k=1$. As in the proof of Proposition \ref{prop:
equivoftalphascutoff}, by the submultiplicativity property $d(mt) \le (2d(t))^{m}
$, together with (\ref{eq: hitTV6}), we have that $\hit_{\alpha}^{(n)}(1-\alpha/2
) \le \mix^{(n)}(\alpha/2) \le \lceil| \log_2 (\alpha /2)|
\rceil) \mix^{(n)}  $.

Conversely, by the laziness assumption, we have that for all $n$,
\begin{equation}
\label{eq: nohitcutoff3}
\hit_{\alpha}^{(n)}(\epsilon/2 ) \ge | \log_2  \epsilon |, \text{ for all
}0<\epsilon<1.
\end{equation}
To see this, consider the case that $X_0^{(n)}=y_n^{(n)}$, for some $y_n
\in \Omega_n$
such that $\pi_n(y_n) \le 1/2 \le 1-\alpha $, and that the first $\lfloor|
\log_2  \epsilon
| \rfloor $ steps of the chain are lazy (i.e.~$y_n=X_1^{(n)}=\cdots = X_{\lfloor|
\log_2  \epsilon
| \rfloor}$).

By (\ref{eq: nohitcutoff2}) in conjunction with (\ref{eq: nohitcutoff3})
we may assume that $\lim_{n \to \infty } \mix^{(n)}=\infty $, as otherwise
there cannot be  $\hit_{\alpha}$-cutoff. By passing to a subsequence, we
may assume further
that there exists some $C>0$ such that $\mix^{(n)} < C \rel^{(n)}$. In particular
$\lim_{n \to \infty}\rel^{(n)}= \infty$ and we may assume without loss of
generality that $(\lambda_2^{(n)})^{\mix^{(n)}} \ge e^{-C} $ for all $n$,
where $\lambda_2^{(n)} $ is the second largest eigenvalue of $P_n$. 

For notational convenience we now suppress the dependence on $n$ from our
notation. 
Let $f_2 \in \R^{\Omega}$ be a non-zero vector satisfying that $P f_2=\lambda_2f_2$.
By considering $-f_2$ if necessary, we may assume that $A:=\{x \in \Omega:f_2
\le 0 \} $ satisfies $\pi(A) \ge 1/2$. Let $x \in \Omega $ be such that $
f_{2}(x)=\max_{y \in \Omega}f_{2}(y) =:L$. Note that $L>0$ since $\mathbb{E}_{\pi}[f_2]=0$.

Consider $N_k:= \lambda_2^{-k } f_{2}(X_{k   })$ and $M_k:=N_{k \wedge T_{A}
} $, where $X_0=x $.   Observe that $(N_k)_{k \ge 0}$ is a martingale and
hence so is $(M_k)_{k \ge 0}$ (w.r.t.~the natural filtration induced by the
chain).
As $M_k \le 0 $ on $\{ T_{A} \le k\}$ and $M_k \le \lambda_2^{-k
} L $ on $\{ T_{A}>k \}$, we get that for all $k >0$
\begin{equation}
\label{eq: SM1}
L=\mathbb{E}_{x}[M_0] = \mathbb{E}_{x}[M_k] \le \mathbb{E}_{x}[\lambda_2^{-k
} L1_{T_{A}>k}] \le \lambda_2^{-k
} L \Pr_{x}[T_{A}>k] .   
\end{equation}
Thus $\Pr_{x}[T_{A}>k] \ge \lambda_2^{k
}$, for all $k$.
Consequently,  for all $a >0 $,
\begin{equation}
\label{eq: nohitcutoff1}
\Pr_{x}[T_{A} > a \mix ] \ge \lambda_2^{ a \mix } \ge e^{-aC} .
\end{equation}
Thus $$\hit_{\alpha}(\epsilon/2) \ge \hit_{1/2}(\epsilon/2) \ge
C^{-1} \mix | \log  \epsilon |, \text{ for any }0< \epsilon < 1 .$$
This, in conjunction with (\ref{eq: nohitcutoff2}), implies that 
 $ \frac{\hit_{\alpha} (\epsilon) }{\hit_{\alpha}(1-\epsilon)}  \ge \frac{|
\log  \epsilon |}{C \lceil \log_2(\alpha/2)
\rceil } $, for all
$0<\epsilon \le \alpha/2$. Consequently, there is no $\hit_{\alpha}$-cutoff.
\end{proof}

\subsection{Proof of Lemma \ref{lem: AF1}}
Now we prove Lemma \ref{lem: AF1}. As mentioned before, the hitting time of a set $A$ starting from stationary initial distribution is well-understood (see \cite{fill2012hitting}; for the continuous-time analog see
\cite{aldous2000reversible},
Chapter 3 Sections 5 and 6.5 or \cite{brown1999interlacing}). Assuming that the chain is lazy, it follows from the theory of complete monotonicity together with some linear-algebra that this distribution is dominated by a distribution which gives mass $\pi(A)$ to $0$, and conditionally on being positive, is distributed as the Geometric distribution with parameter $\frac{\pi(A)}{t_{\mathrm{ rel}}}$.
 Since the existing literature lacks simple treatment of this fact (especially for the discrete-time case) we now prove it for the sake of completeness. We shall prove this fact without assuming laziness. Although without assuming laziness the distribution of $T_A$ under $\Pr_{\pi} $ need not be completely monotone, the proof is essentially identical as in the lazy case.

For any non-empty $ A \subset \Omega $, we write $\pi_{A}$
for the distribution
of $\pi$ conditioned on $A$. That is,  $\pi_{A}(\cdot):=\frac{\pi(\cdot) 1_{\cdot
\in A}}{\pi(A)}$.

\begin{lemma}
\label{lem: CMlem}
Let $(\Omega,P,\pi)$ be a reversible irreducible finite chain. Let $A \subsetneq \Omega$ be non-empty. Denote
its complement by $B$ and write $k=|B|$. Consider the sub-stochastic matrix
$P_{B}$, which is the restriction of $P$ to $B$. That is $P_{B}(x,y):=P(x,y)$
for $x,y \in B$. Assume that $P_{B}$ is irreducible, that is, for any $x,y
\in B$, exists some $t \ge 0$ such that $P_{B}^{t}(x,y)>0$. Then
\begin{itemize}
\item[(i)] $P_{B}$ has $k$ real eigenvalues $1-\frac{\pi(A)}{t_{\mathrm{
rel}}}\ge \gamma_1 > \gamma_2 \ge \cdots
\ge \gamma_{k} \ge -\gamma_1$. 
\item[(ii)] There exist some non-negative $a_1,\ldots,a_k$ such that for any $t \ge 0$ we have that
\begin{equation}
\label{eq: discreteCM1}
\Pr_{\pi_{B}}[T_{A} > t] = \sum_{i=1}^{k}a_i \gamma_i^t.
\end{equation}
\item[(iii)]
\begin{equation}
\label{eq: CM-1}
\Pr_{\pi_{B}}[T_{A} > t] \le \left(1- \frac{\pi(A)}{t_{\mathrm{rel}}}
\right)^{t} \le \exp \left(- \frac{t \pi(A)}{t_{\mathrm{rel}}} \right), \text{ for any }t \ge 0. 
\end{equation}
\end{itemize}
\end{lemma}
\begin{proof}
We first note that (\ref{eq: CM-1}) follows immediately from (\ref{eq: discreteCM1}) and (i). Indeed, plugging $t=0$ in (\ref{eq: discreteCM1}) yields that $\sum_i a_i=1$. Since by (i), $|\gamma_i| \le \gamma_1 \le 1-\frac{\pi(A)}{t_{\mathrm{rel}}}  $ for all $i$, (\ref{eq: discreteCM1}) implies that $\Pr_{\pi_{B}}[T_{A} \ge t] \le \gamma_1^t\le \left(1- \frac{\pi(A)}{t_{\mathrm{rel}}}
\right)^{t}  $ for all $t \ge 0$.

We now prove (i). Consider the following inner-product on $\R^B$ defined
by $\langle f,g \rangle_{\pi_{B}}:=\sum_{x \in B}\pi_{B}(x)f(x)g(x)$. Since $P$ is reversible, $P_{B}$ is self-adjoint w.r.t.~this inner-product.
Hence indeed $P_{B}$ has $k$ real eigenvalues $ \gamma_1> \gamma_2 \ge \cdots
\ge \gamma_{k} $ and there is a basis of $\R^B$, $g_1,\ldots,g_k$ of orthonormal
vectors w.r.t.~the aforementioned inner-product, such that $P_B g_i =\gamma_i
g_i$ ($i \in [k]$). By the Perron-Frobenius Theorem $\gamma_1> 0$ and $\gamma_1
\ge - \gamma_k$. 

The claim that  $1-\gamma_1 \ge \frac{\pi(A)}{t_{\mathrm{ rel}}}
$, follows by the Courant-Fischer
characterization of the spectral gap and comparing the Dirichlet forms of
$\langle \cdot , \cdot \rangle_{\pi_{B}}$, and $\langle \cdot , \cdot \rangle_{\pi}$
(c.f.~Lemma 2.7
in \cite{ding2010total} or Theorem 3.3 and Corollary 3.4 in Section 6.5 of
 Chapter 3 in \cite{aldous2000reversible}). This concludes the proof of part
(i). We now prove part (ii).

  By summing over all  paths of length $t$ which are contained in $B$ we get that
\begin{equation}
\label{eq: CM7}
\Pr_{\pi_B}[T_{A} >t]=\sum_{x,y
\in B} \pi_B(x) P_B^t(x,y).
\end{equation}
By the spectral representation (c.f.~Lemma
12.2 in \cite{levin2009markov} and Section 4 of Chapter 3 in \cite{aldous2000reversible})
for any $x,y \in B$ and $t \in \N$ we have that $P_{B}^{t}(x,y)=\sum_{i=1}^{k}\pi_{B}(y)g_{i}(x)g_{i}(y)
\gamma_{i}^t $. So by (\ref{eq: CM7})
\begin{equation*}
\begin{split}
\Pr_{\pi_B}[T_{A} >t]=\sum_{x,y
\in B} \pi_B(x) \sum_{i=1}^{k} \pi_{B}(y)g_{i}(x)g_{i}(y)
\gamma_{i}^t= \sum_{i=1}^{k} \left(\sum_{x
\in B} \pi_B(x)g_{i}(x) \right)^{2}
\gamma_{i}^t. 
\end{split}
\end{equation*}
\end{proof}
\emph{Proof of Lemma \ref{lem: AF1}.}
We first note that (\ref{eq: CM2}) follows easily from
(\ref{eq: CM1}). For the first inequality in (\ref{eq: CM2}) denote $B:=B(A,w,\alpha)=\left\{y:\Pr_{y}
\left[T_{A} > \left\lceil \frac{t_{\mathrm{rel}}w }{\pi(A)}\right\rceil
\right] \ge \alpha  \right\}$ and $t=t(A,w):= \left\lceil \frac{t_{\mathrm{rel}}w
}{\pi(A)}\right\rceil $. Then by (\ref{eq: CM1}) $$\alpha \pi (B) \le \pi(B)\Pr_{\pi_B}[T_{A}
>t] \le \Pr_{\pi}[T_{A} >t] \le \pi(A^c)  \exp \left(- \frac{t \pi(A)}{t_{\mathrm{rel}}}
\right) \le \pi(A^c)e^{-w}.$$
We now prove (\ref{eq: CM1}). Denote the connected components of $A^c:=\Omega \setminus A
$ by $\{C_1,\ldots,C_k \}$. Denote the complement of $C_i$ by $C_i^c$.  
By (\ref{eq: CM-1}) we have that
\begin{equation*}
\label{eq: CM4}
\begin{split}
&\Pr_{\pi}[T_{A} > t] = \sum_{i=1}^{k}\pi(C_i)\Pr_{\pi_{C_i}}^{}[T_{A}
> t] = \sum_{i=1}^{k}\pi(C_i)\Pr_{\pi_{C_i}}^{}[T_{C_{i}^c}
> t] \le   \\ & \sum_{i=1}^{k}\pi(C_i)\exp
\left(-\frac{t \pi(C_{i}^c)}{t_{\mathrm{rel}}} \right) \le   \sum_{i=1}^{k}\pi(C_i)\exp
\left(-\frac{t \pi(A)}{t_{\mathrm{rel}}} \right)= \pi(A^c) \exp
\left(-\frac{t \pi(A)}{t_{\mathrm{rel}}} \right). \qed
\end{split}
\end{equation*}
\vspace{2mm}
\section{Continuous-time}
\label{sec: ct}
In this section we explain the necessary adaptations in the proof of Proposition \ref{prop: TVbound0} for the continuous-time case. We fix some finite, irreducible, reversible chain $(\Omega,P,\pi)$. For notational convenience, exclusively for this section, we shall denote the transition-matrix of $(X_{k}^\mathrm{NL})_{k
\ge 0}$, the non-lazy version of the discrete-time chain, by $P$, and that of the lazy version of the chain by $P_L:=(P+I)/2$.

We denote the eigenvalues of $P$ by $1=\lambda_1^{\mathrm{ct}} >\lambda_2^{\mathrm{ct}} \ge \cdots \le \lambda_{|\Omega|}^{\mathrm{ct}} \ge -1$ and that of $P_L$ by $1=\lambda_1^{L}
>\lambda_2^{L} \ge \cdots \le \lambda_{|\Omega|}^{L}
\ge -1 $ (where $1+\lambda_i^{\mathrm{ct}}=2\lambda_{i}^{L}$). We denote $t_{\mathrm{rel}}^{\mathrm{ct}}:=(1-\lambda_2^{\mathrm{ct}})^{-1} $ and $t_{\mathrm{rel}}^{L}:=(1-\lambda_2^{L})^{-1} $. We identify $H_t$ with the operator $H_t:L^2(\R^{\Omega},\pi)
\to L^2(\R^{\Omega},\pi)$, defined by $H_t f(x)= \mathbb{E}_{x}[f(X_t^{\mathrm{ct}})]
$. The spectral decomposition in continuous time takes the following form. If $f_1,\ldots,f_{|\Omega|}$ is an orthonormal basis such that $Pf_i:=\lambda_i^{\mathrm{ct}}
f_i$ for all $i$, then $H_t g=\mathbb{E}_{\pi}H_t g+ \sum_{i=2}^{|\Omega|} \langle g,f_{i}\rangle_{\pi}e^{-(1-\lambda_i^{\mathrm{ct}})t} f_i$, for all $g \in \R^{\Omega}$ and $t \ge 0$. Thus the $L^2$-contraction Lemma takes the following form in continuous-time (see e.g.~Lemma 20.5 in \cite{levin2009markov}):
\begin{equation}
\label{eq: L2contractionct}
\Var_{\pi}H_tf \le e^{-2t/t_{\mathrm{rel}}^{\mathrm{ct}}}\Var_{\pi}f, \text{ for any }f \in \R^{\Omega}, \text{ for any }t \ge 0.
\end{equation}
Starr's inequality holds also in continuous-time (\cite{starr1966operator}
Proposition 3) and takes the following form. Let $f \in \R^{\Omega}$. Define the \textbf{\textit{continuous-time maximal function}}
as $f_{\mathrm{ct}}^{*}(x):=\sup_{t \ge 0}|H_{t}f (x)|$. Then
\begin{equation}
\label{eq: ctmaxergthm}
\|f_{\mathrm{ct}}^{*} \|_{2} \le 2 \|f \|_2.
\end{equation}
We note that our proof of Theorem \ref{thm: maxergodic} can easily be adapted to the continuous-time case.

For any $A \subset \Omega$ and $s \in \R_+$, set $\rho(A):=\sqrt{\pi(A)(1-\pi(A))} $ and $\sigma_{s}^{\mathrm{ct}}:=\rho(A)e^{s/t_{\mathrm{rel}}^{\mathrm{ct}}} $. Define
$$G_{s}^{\mathrm{ct}}(A,m):=\left\{ y: |\h_y^k(A)-\pi(A)| < m\sigma_s^{\mathrm{ct}} \text{
for all }k \ge s \right\}, $$
Then similarly to Corollary \ref{cor: maxergcor}, combining (\ref{eq: L2contractionct}) and (\ref{eq: ctmaxergthm}) (in continuous-time there is no need to treat odd and even times separately) yields
\begin{equation}
\label{eq: Goodsetct}
\pi(G_{s}^{\mathrm{ct}}(A,m)) \ge 1-4/m^2, \text{ for all }A \subset \Omega,\, s\ge 0 \text{ and }m>0. 
\end{equation}
The proof of Corollary \ref{cor: hitprob} carries over to the continuous-time case (where everywhere in (\ref{eq: hitprob1})-(\ref{eq: hitTV2}),
$t_{\mathrm{mix}}$ and $\mathrm{hit}$ are replaced by $t_{\mathrm{mix}}^{\mathrm{ct}}$
and $\mathrm{hit}^{\mathrm{ct}}$, respectively,  and all ceiling signs are omitted), using (\ref{eq: Goodsetct}) rather than (\ref{eq: G_t}) as in the discrete-time case.

In Lemma \ref{lem: CMlem}, we showed that for any non-empty $A \subsetneq \Omega$ such that $P_{A^c}$ is irreducible, $P_{A^{c}}$ has $k$ real eigenvalues $1-\frac{\pi(A)}{t_{\mathrm{
rel}}^{\mathrm{ct}}}\ge \gamma_1 > \gamma_2 \ge \cdots
\ge \gamma_{k} \ge -\gamma_1$ and that there exists some convex combination $a_1,\ldots,a_k$ such that $\Pr_{\pi_{A^c}}[T_A > t]= \sum_{i=1}^{k}a_i\gamma_i^t \le (1-\frac{\pi(A)}{t_{\mathrm{
rel}}^{\mathrm{ct}}})^{t}\le \exp \left(-\frac{t\pi(A)}{t_{\mathrm{
rel}}^{\mathrm{ct}}}\right) $, for any $t \ge 0$. Repeating the argument while using the spectral decomposition of $(H_{A^c})_t$ (the restriction of $H_t$ to $A^c$) in continuous-time, rather than the discrete time spectral decomposition, yields that  $\h_{\pi_{A^c}}[T_A^{\ct} > t]= \sum_{i=1}^{k}a_i e^{-(1-\gamma_i)t}
\le \exp ({-\frac{t\pi(A)}{t_{\mathrm{
rel}}^{\mathrm{ct}}}})$, for any $t \ge 0$. 
Consequently, as in Lemma \ref{lem: AF1}, 
$B_{\mathrm{ct}}(A,w,\alpha):=\left\{y:\h_{y}
\left[T_{A}^{\ct} \ge \frac{t_{\mathrm{rel}}^{\mathrm{ct}}w }{\pi(A)}
\right] \ge \alpha  \right\}$ satisfies that 
\begin{equation}
\label{eq: CM2ct}
\pi \left(B_{\mathrm{ct}}(A,w,\alpha) \right) \le \pi(A^c)   e^{-w}\alpha^{-1} \text{, for all } w \ge 0 \text{ and } 0< \alpha \le 1.
\end{equation}
Using (\ref{eq: CM2ct}) rather than (\ref{eq: CM2}), Corollary \ref{prop: hitpqinequalities} is extended to the continuous-time case. Namely, for any reversible irreducible finite chain and any $0< \epsilon < \delta  <
1$,
\begin{equation}
\label{eq: thitinequality1ct}
\mathrm{hit}_{\beta }^{\mathrm{ct}}(\delta) \le \mathrm{hit}_{\alpha }^{\mathrm{ct}}(\delta)
\le \mathrm{hit}_{\beta }^{\mathrm{ct}}(\delta- \epsilon)+ \alpha^{-1} t_{\mathrm{rel}}^{\mathrm{ct}}
\log \left( \frac{1-\alpha}{(1-\beta)\epsilon} \right), \text{
for any } 0 < \alpha \le \beta < 1.
\end{equation}
Finally, using (\ref{eq: thitinequality1ct}), rather than (\ref{eq: thitinequality1}) as in the discrete-time case, together with the version of Corollary
\ref{cor: hitprob} for the continuous-time chain, the proof of Proposition \ref{prop: TVbound0} for the continuous-time case is concluded in the same manner as the proof in the discrete-time case.
\section{Trees}
\label{sec: trees}

We start with a few definitions. Let $T:=(V,E)$ be a finite tree. Throughout the section we fix some lazy Markov chain, $(V,P,\pi)$, on a finite tree $T:=(V,E)$. That is, a chain
with stationary distribution $\pi$ and state space $V$ such that $P(x,y)>0$ iff $\{x,y\} \in E$ or $y=x$ (in which case, $P(x,x)\ge 1/2$). Then $P$
is reversible by Kolmogorov's cycle condition.

Following \cite{peres2011mixing}, we call a vertex $v \in V$ a \emph{\textbf{central-vertex}} if each connected component of $T \setminus \{v\}$ has stationary probability at most 1/2. A central-vertex always exists (and there may be at most two central-vertices). Throughout, we fix a central-vertex $o$ and call it the \emph{\textbf{root}} of the tree. We denote a (weighted) tree with root $o$ by $(T,o)$. 

Loosely speaking, the analysis below shows that a chain on a tree satisfies the product condition iff it has a ``global bias" towards $o$. A non-intuitive result is that one can construct such unweighed trees \cite{peres2013total}. 

The root induces a partial order $\prec $ on $V$,  as follows. For every $u \in V$, we denote the shortest path between $u$ and $o$ by $\ell(u)=(u_0=u,u_1,\ldots,u_k=o)$. We call $f_{u}:=u_1$ the \emph{\textbf{parent}} of $u$ and denote $\mu_u:=P(u,f_{u})$. We say that $u' \prec u$ if $u' \in \ell(u)$. Denote $W_{u}:=\{v: u \in \ell(v) \}$.
Recall that for any $\emptyset \neq A \subset V $, we write $\pi_{A}$ for the distribution of $\pi$ conditioned on $A$,  $\pi_{A}(\cdot):=\frac{\pi(\cdot) 1_{\cdot \in A}}{\pi(A)}$.

%
A key observation is that starting from the central vertex $o$ the chain mixes rapidly (this follows implicitly from the following ananlysis). Let $T_o$ denote the hitting time of the central vertex. We define the mixing parameter $\tau(\epsilon)$ for $\epsilon\in (0,1)$ by $$\tau_o(\epsilon):=\min\{t:\Pr_{x}[T_o>t]\leq \epsilon~\forall x\in \Omega\}.$$
We show that up to terms of the order of the relaxation-time (which are negligible under the product condition) $\tau_o(\cdot)$ approximates ${\rm hit}_{1/2}(\cdot)$ and then using Proposition \ref{prop: TVbound0}, the question of cutoff is reduced to showing concentration for the hitting time of the central vertex. Below we make this precise.  
%
%
%


\begin{lemma}
\label{lem: TVboundstrees}
Denote $s_{\delta}:=\lceil 4t_{\rm rel} |\log (4 \delta /9 )|\rceil$. Then 
\begin{equation}
\label{e:hittocentral1}
\tau_o(\epsilon)\leq {\rm hit}_{1/2}(\epsilon) \leq \tau_o(\epsilon-\delta)+
s_{\delta}, \text{ for every } 0< \delta< \epsilon < 1.
\end{equation}
\end{lemma}
\begin{proof}
First observe that by the definition of central vertex, for any $x\in V$, $x\neq o$ there exists a set $A$ with $\pi(A)\geq \frac{1}{2}$ such that the chain starting at $x$ cannot hit $A$ without hitting $o$. Indeed, we can take $A$ to be the union of $\{o\}$ and all components of $T\setminus \{o\}$ not containing $x$. The first inequality in (\ref{e:hittocentral1}) follows trivially from this. 

To establish the other inequality, fix $A\subseteq V$ with $\pi(A)\geq \frac{1}{2}$, $x \in V$ and some $0<\delta < \epsilon <1$. It follows using Markov
property and the definition of $\tau_o(\epsilon-\delta) $ that $$\Pr_x[T_A>\tau_o(\epsilon-\delta)+s_{\delta}]\leq
\Pr_x[T_o>\tau_o(\epsilon-\delta)]+\Pr_o[T_A>s_{\delta}] \le \epsilon -\delta +\Pr_o[T_A>s_{\delta}] .$$ Hence it suffices to show that $\Pr_o[T_{A}>s_{\delta}] \le \delta$. If $o\in A$ then $\Pr_o[T_A>s_{\delta}]=0$, so without loss of generality assume $o\notin A$. It is easy to see that we can partition  $T\setminus \{o\}=T_1\cup T_2$  such that both $T_1$ and $T_2$ are unions of components of $T\setminus \{o\}$ and $\pi(T_1),\pi(T_2)\leq 2/3$. For $i=1,2$, let $A_i:=A\cap T_i$ and without loss of generality let us assume $\pi(A_1)\geq \frac{1}{4}$. Let $B=T_2\cup \{o\}$. Clearly the chain started at any $x\in B$ must hit $o$ before hitting $A_{1}$. Hence
\begin{equation}
\label{e:hittocentral2}
\Pr_{o}[T_A>s_{\delta}]\leq \Pr_{o}[T_{A_1}>s_{\delta}]\leq \Pr_{\pi_B}[T_{A_1}>s_{\delta}]\leq \pi(B)^{-1}\Pr_{\pi}[T_{A_1}>s_{\delta}] 
\end{equation}

Using $\pi(A_1)\geq \frac{1}{4}$, $\pi(B)\geq \frac{1}{3}$ it follows from (\ref{eq: CM1}) that $\pi(B)^{-1}\Pr_{\pi}[T_{A_1}>s_{\delta}] \le \delta $.      
\end{proof}

In light of Lemma \ref{lem: TVboundstrees} and Proposition \ref{prop: TVbound0}, in order to show that in the setup of Theorem \ref{thm: treescutoff} (under the product condition) cutoff occurs it suffices to show that  $\tau_{o_n}^{(n)}(\epsilon)-\tau_{o_n}^{(n)}(1-\epsilon) =o(t_{\mathrm{mix}}^{(n)})$, for any $\epsilon \in (0,1/4]$. We  actually show more than that. Instead of identifying the ``worst" starting position $x$ and proving that $T_{o}$ is concentrated under $\Pr_x$, we shall show that for any $x,y \in V_n $ such that $y \prec x$ and $\mathbb{E}_x[T_y]=\Theta(t_{\mathrm{mix}}^{(n)}) $, $T_y$ is concentrated under $\Pr_x$, around $\mathbb{E}_x[T_y]$, with deviations of order $\sqrt{t_{\mathrm{rel}}^{(n)} t_{\mathrm{mix}}^{(n)}} $. This shall follow from Chebyshev inequality, once we establish that $\Var_{x}[T_{y}] \le 4t_{\mathrm{rel}}\mathbb{E}_x[T_y]$.

Let $(v_0=x,v_1,\ldots,v_{k}=y)$ be the path from $x$ to $y $ ($y \prec x$). Define $\tau_i:=T_{v_{i}}-T_{v_{i-1}}$. Then by the tree structure, under $\Pr_x$ we have that $T_{y}=\sum_{i=1}^{k}\tau_i$ and that $\tau_1,\ldots,\tau_k$ are independent. This reduces the task of bounding $\Var_{x}[T_{y}] $ from above, to the task of estimating $\Var_{v_{i}}[T_{v_{i+1}}]=\Var_{v_{i}}[T_{f_{v_{i}}}] $ from above for each $i$.  
\begin{lemma}
\label{lem: tree1}
For any vertex $u\neq o$ we have that
\begin{equation}
\label{eq: crossingtime}
t_{u}:=\mathbb{E}_{u}[T_{f_{u}}]=\frac{\pi(W_u)}{\pi(u)\mu_u} \text{ and } r_u:=\mathbb{E}_{u}[T_{f_{u}}^2]=2t_u \mathbb{E}_{\pi_{W_u}}[T_{f_{u}}]-t_u \le 4t_u t_{\mathrm{rel}}.
\end{equation}
\end{lemma}
The assertion of Lemma \ref{lem: tree1} follows as a particular case of Proposition \ref{prop: Kac} at the end of this section. 
\begin{corollary}
\label{cor: concentrationfortrees}
Let $x,y \in V $ be such that $y \preceq x$ and $c \ge 0$. Denote $\sigma_{x,y}:=\sqrt{4 \mathbb{E}_{x}[T_{y}] t_{\mathrm{rel}}}$. Then
\begin{equation}
\label{eq: VarofTybeta}
\Var_{x}[T_{y}] \le \sigma_{x,y}^2,
\end{equation}
and 
\begin{equation}
\label{eq: VarofTybeta2}
\Pr_{x}[T_{y}  \ge \mathbb{E}_{x}[T_{y}]+ c \sigma_{x,y} ] \le \frac{1}{1+c^2} \text{ and } \Pr_{x}[T_{y}  \le \mathbb{E}_{x}[T_{y}]- c \sigma_{x,y}
] \le  \frac{1}{1+c^2}.
\end{equation}
In particular, if $(V_n,P_n,\pi_n)$ is a sequence of lazy Markov chains on trees $(T_n,o_n)$ which satisfies the product condition, and $x_n,y_n \in V_n$ satisfy that $y_n \prec x_n$ and $\mathbb{E}_{x_n}[T_{y_{n}}]/t_{\mathrm{rel}}^{(n)}\to \infty$, then for any $\epsilon>0$ we have that
\begin{equation}
\label{eq: VarofTybeta3}
\lim_{n \to \infty} \Pr_{x_{n}}[|T_{y_{n}}-\mathbb{E}_{x_n}[T_{y_{n}}]| \ge \epsilon \mathbb{E}_{x_n}[T_{y_{n}}] ]=0.
\end{equation}
\end{corollary}
\begin{proof}
We first note that (\ref{eq: VarofTybeta2}) follows from (\ref{eq: VarofTybeta}) by the one-sided Chebyshev inequality. Also, (\ref{eq: VarofTybeta3}) follows immediately from (\ref{eq: VarofTybeta2}). We now prove (\ref{eq: VarofTybeta}). Let $(v_0=x,v_1,\ldots,v_{k}=y)$ be the path from $x$ to $y$.
Define $\tau_i:=T_{v_{i}}-T_{v_{i-1}}$. Then by the tree structure, under
$\Pr_x$, we have that $T_{y}=\sum_{i=1}^{k}\tau_i$ and that $\tau_1,\ldots,\tau_k$
are independent. Whence, by (\ref{eq: crossingtime}) we get that
$$\Var_{x}[T_{y}] = \sum_{i=1}^{k} \Var_{x}[\tau_i]= \sum_{i=1}^{k} \Var_{v_{i-1}}[T_{v_{i}}] \le  \sum_{i=1}^{k}
\mathbb{E}_{v_{i-1}}[T_{v_{i}}^2] \le 4 t_{\mathrm{rel}} \sum_{i=1}^{k}
\mathbb{E}_{v_{i-1}}[T_{v_{i}}]=\sigma_{x,y}^2.$$
This completes the proof.
\end{proof}
\begin{lemma}
\label{lem: hittingsetsofsize3/4}
If $(V,P,\pi)$ is a lazy chain on a (weighted) tree $(T,o)$ then
\begin{equation}
\label{eq: trivial}
\mathbb{E}_x[T_{o}] \le 4t_{\mathrm{mix}}, \text{ for all }x \in V.
\end{equation}  
\end{lemma}
\begin{proof}
Fix some $x \in V $. Let $C_x$ be the component
of $T\setminus \{o\}$ containing $x$. Denote $B:=V \setminus C_x $. Consider
$\tau_{B}:=\inf\{k
\in \N : X_{kt_{\mathrm{mix}}}\in B \}$. Clearly, $T_{o} \le \tau_{B} t_{\mathrm{mix}}$.
Since $\pi(B)\geq 1/2$, by the Markov property and the definition of the
total variation distance,
the distribution of $\tau_{B}$ is stochastically dominated by the Geometric
distribution with parameter $1/2-1/4=1/4$. Hence $\mathbb{E}_{x}[T_{0}] =\mathbb{E}_{x}[T_{B}]
\le t_{\mathrm{mix}}\mathbb{E}_{x}[\tau_{B}] \le  4t_{\mathrm{mix}}$.
\end{proof} 
\begin{corollary}
\label{cor: concentrationfortrees1}
In the setup of Lemma \ref{lem: tree1}, for any $x \in V$ denote $t_{x}:=\mathbb{E}_{x}[T_o]$. Fix $\epsilon\in (0,\frac{1}{4}]$, Denote $$\rho:=\max_{x \in V }t_{x},\text{ and } \kappa_{\epsilon}:= \sqrt{ 4\epsilon^{-1} \rho t_{\mathrm{rel}}}, \text{ then} $$ 
\begin{equation}
\label{eq: Chebyshevbounds}
\rho \le 4t_{\mathrm{mix}},\, \tau_{o}(1-\epsilon) \ge \rho- \kappa_{\epsilon} \text{ and } \tau_{o}(\epsilon) < \rho+ \kappa_{\epsilon}.
\end{equation}
\end{corollary}

\begin{proof}
By (\ref{eq: trivial}) $\rho \le 4t_{\mathrm{mix}}$. Denote $\sigma:= \sqrt{4\rho t_{\mathrm{rel}}}$ and $c_{\epsilon}:=\sqrt{
\epsilon^{-1} -1} $. Take $x\in V \setminus \{ o \} $. 
By (\ref{eq: VarofTybeta}) $\sigma_{x,o}^2:=\Var_{x}[T_{o}] \le \sigma^2$.  The assertion of the corollary now follows from (\ref{eq: VarofTybeta2}) by noting that $c_{\epsilon}\sigma\leq \kappa_{\epsilon}$.
\end{proof}
Now we are ready to prove Theorem \ref{thm: treescutoff}.
\begin{proof}[Proof of Theorem \ref{thm: treescutoff}] 
Fix $\epsilon \in (0,\frac{1}{4}]$. It follows from  (\ref{eq: introTVbound1})
and (\ref{eq: introTVbound2}) that 
\begin{equation}
\label{e:treeproof1}
t_{\rm mix}(\epsilon)-t_{\rm mix}(1-\epsilon) \leq {\rm hit}_{1/2}(\epsilon/2)-{\rm
hit}_{1/2}(1-\epsilon/2)+  t_{\rm rel} (3| \log  \epsilon | +\log 4) +2.
\end{equation}
Using Lemma \ref{lem: TVboundstrees} with $(\epsilon, \delta)$ there replaced
by $(\epsilon/2,\epsilon/4)$ it follows that 
\begin{equation}
\label{e:treeproof2}
{\rm hit}_{1/2}(\epsilon/2)- {\rm hit}_{1/2}(1-\epsilon/2)\leq \tau_o(\epsilon/4)-\tau_o(1-\epsilon/2)+s_{\epsilon/4}
\end{equation}
where $s_{\epsilon/4}$ is as in Lemma \ref{lem: TVboundstrees}. It follows
from (\ref{e:treeproof1}), (\ref{e:treeproof2}) and (\ref{eq: Chebyshevbounds})
that 
\begin{equation}
\label{e:treeproof3}
t_{\rm mix}(\epsilon)-t_{\rm mix}(1-\epsilon)\leq \kappa_{\epsilon/4}+\kappa_{\epsilon/2}+
t_{\rm rel}(7|\log \epsilon|+4 \log 9 - 3 \log 4) +3.
\end{equation}
It follows from (\ref{eq: Chebyshevbounds}) that $\kappa_{\epsilon/4}+\kappa_{\epsilon/2}
\leq 14\sqrt{\epsilon^{-1}t_{\rm rel}t_{\rm mix}}$. 
For any irreducible Markov
chain on $n>1$ states we have that $\lambda_2 \ge -\frac{1}{n-1}$ (\cite{aldous2000reversible},Chapter
3 Proposition 3.18). Hence for a lazy chain with at least 3 states
we have that $t_{\mathrm{rel}} \ge 4/3$ and so by (\ref{eq: t_relintro})
$t_{\mathrm{rel}} \le 6 (t_{\mathrm{rel}} -1)\log 2 \le 6t_{\mathrm{mix}}
$. Using the fact that $|\log \epsilon| \le \frac{2}{e\sqrt{\epsilon}} $
for every $0<\epsilon \le 1/4 $, it follows that $7t_{\rm rel}|\log \epsilon|\leq
7\sqrt{6}\frac{2}{e} \sqrt{\epsilon^{-1}t_{\rm rel}t_{\rm mix}}\leq 13\sqrt{\epsilon^{-1}t_{\rm
rel}t_{\rm mix}}$. As $\sqrt{6}(4 \log 9 - 3 \log 4)<12 $ and $\sqrt{ \epsilon^{-1}}
\geq 2$ we also have that $\rel (4 \log 9 - 3 \log 4)+3 \le 8 \sqrt{\epsilon^{-1}t_{\rm
rel}t_{\rm mix}} $. Plugging these estimates in  (\ref{e:treeproof3}) completes
the proof of the theorem.
\end{proof}
As promised earlier, the following proposition implies the assertion of Lemma \ref{lem: tree1}.
For any set $A \subset \Omega $, we define $\psi_{A^{c}}\in\mathscr{P}(A^{c})$
as $\psi_{A^{c}}(y):=\Pr_{\pi_{A}}[X_{1}=y \mid X_1 \in A^c ]$. For $A \subset
\Omega$, we denote $T_A^+:=\inf \{t \ge 1:X_t \in A \}$ and $\Phi(A):=\frac{\sum_{a
\in A,b \in A^c}\pi(a)P(a,b)}{\pi(A)}=\Pr_{\pi_A}[X_1 \notin A] $. Note
that
\begin{equation}
\label{eq: Q(A,A^c)=Q(A^c,A)}
\pi(A)\Phi(A)=\sum_{a
\in A,b \in A^c}\pi(a)P(a,b)=\sum_{a
\in A,b \in A^c}\pi(b)P(b,a)=\pi(A^c)\Phi(A^c).
\end{equation}
This is true even without reversibility, since the second term (resp.~third
term) is the asymptotic frequency of transitions from $A$ to $A^c$ (resp.~from
$A^c$ to $A$).
\begin{proposition}
\label{prop: Kac}
Let $(\Omega,P,\pi)$ be a finite irreducible reversible Markov chain. Let
$A \varsubsetneq\ \Omega$ be non-empty. Denote the complement of $A$ by $B$.
Then
\begin{equation}
\label{eq: Kac}
\Pr_{\pi_{B}}[T_{A}=t]/\Phi(B)=\Pr_{\psi_{B}}[T_{A} \ge t], \text{ for any
}t \ge 1.
\end{equation}
Consequently,
\begin{equation}
\label{eq: kac1}
\mathbb{E}_{\psi_{B}}[T_{A}]=\frac{1}{\Phi(B)} \text{ and } \mathbb{E}_{\psi_{B}}[T_{A}^{2}]=\mathbb{E}_{\psi_{B}}[T_{A}]
\left(2\mathbb{E}_{\pi_{B}}[T_{A}]-1 \right) \le \frac{2\mathbb{E}_{\psi_{B}}[T_{A}]t_{\mathrm{rel}}
}{\pi(A)}.
\end{equation}
\end{proposition}   
\begin{proof}
We first note that the inequality $2\mathbb{E}_{\psi_{B}}[T_{A}]\mathbb{E}_{\pi_{B}}[T_{A}]\le
\frac{2\mathbb{E}_{\psi_{B}}[T_{A}]t_{\mathrm{rel}}
}{\pi(A)}$ follows from the second inequality in (\ref{eq: CM2}) (this is the only part of the proposition which relies upon reversibility). 

Summing (\ref{eq: Kac}) over $t$ yields the first equation in (\ref{eq: kac1}). Multiplying both
sides of (\ref{eq: Kac}) by $2t-1$ and summing over $t$ yields the second equation in (\ref{eq: kac1}). We now prove (\ref{eq: Kac}). Let $t \ge 1$. Then
\begin{equation*}
\begin{split}
&\pi(B)\Pr_{\pi_{B}}[T_{A}=t]=\Pr_{\pi}[T_{A}=t]=\Pr_{\pi}[T_{A}^{+}=t+1]=\Pr_{\pi}[X_{1}\notin
A,\ldots,X_t \notin A,X_{t+1}\in A ] \\ &= \Pr_{\pi}[X_{1}\notin
A,\ldots,X_t \notin A ] - \Pr_{\pi}[X_{1}\notin
A,\ldots,X_t \notin A,X_{t+1}\notin A ] \\ &=\Pr_{\pi}[X_{1}\notin
A,\ldots,X_t \notin A ] - \Pr_{\pi}[X_{0}\notin
A,\ldots,X_t \notin A ]=\Pr_{\pi}[X_0 \in A, X_{1}\notin
A,\ldots,X_t \notin A ] \\ &= \pi(A)\Phi(A)\Pr_{\psi_{B}}[X_{0}\notin
A,\ldots,X_{t-1} \notin A]=\pi(A)\Phi(A)\Pr_{\psi_{B}}[T_{A} \ge t],
\end{split}
\end{equation*}
which by (\ref{eq: Q(A,A^c)=Q(A^c,A)}) implies (\ref{eq: Kac}).
\end{proof}
\section{Refining the bound for trees}
The purpose of this section is to improve the concentration estimate (\ref{eq: VarofTybeta2}). As a motivating example, consider a lazy nearest neighbor random walk on a path of length $n$ with some fixed bias to the right. For concreteness, say, $\Omega_{n}:=\{1,2,\ldots,n\}$, $P_{n}(i,i)=1/2$, $P_{n}(i,i-1)=1/8$ and $P_{n}(i,i+1)=3/8$ for all $1<i<n$. Then $t_{\mathrm{mix}}^{(n)} =4n(1 + o(1) )$ and $t_{\mathrm{rel}}^{(n)}=\Theta(1)$.

In this case, there exists some constant $c_{1}>0$ such that for any $\lambda>0$
we have that $\Pr_{1}[|T_{n}-4n| \ge \lambda \sqrt{n} ] \le 2 e^{-c_{1}\lambda^2}$.
Observe that $\sqrt{t_{\mathrm{mix}}^{(n)}t_{\mathrm{rel}}^{(n)}} =\Theta
(\sqrt{n})$. Hence there exists some constant $c_2$ such that $\Pr_{1}\left[|T_{n}-4n|
\ge \lambda \sqrt{t_{\mathrm{mix}}^{(n)}t_{\mathrm{rel}}^{(n)}}\right] \le
2 e^{-c_{2}\lambda^2}$. Using Proposition \ref{prop: TVbound0}, it is not
hard to show that this implies that $t_{\mathrm{mix}}^{(n)}(\epsilon) \le
t_{\mathrm{mix}}^{(n)}+ c_3 \sqrt{t_{\mathrm{mix}}^{(n)}t_{\mathrm{rel}}^{(n)}|\log
\epsilon |} $ and that $t_{\mathrm{mix}}^{(n)}(1-\epsilon) \ge t_{\mathrm{mix}}^{(n)}-
c_3 \sqrt{t_{\mathrm{mix}}^{(n)}t_{\mathrm{rel}}^{(n)}|\log \epsilon
|}
$. It is also not hard to verify that in this case (and also in many other
examples of birth and death chains) this is  sharp.

In Lemma \ref{lem: gaussiantails} we show that for any lazy Markov chain
on a tree $T=(V,E,o)$ and any $x \in V$, we have that $\Pr_{x}[|T_{o}-\mathbb{E}_{x}[T_{o}]|
\ge \lambda \sqrt{\mathbb{E}_{x}[T_{o}]t_{\mathrm{rel}}} ] \le 2e^{-c_4 \lambda^2}$.
Besides  being of independent interest, using Proposition \ref{prop: TVbound0},
one can deduce from Lemma \ref{lem: gaussiantails} that under the product
condition,
\begin{equation}
\label{eq: subgausiantrees}
\frac{ t_{\mathrm{mix}}^{(n)}(\epsilon)-t^{(n)}_{\mathrm{mix}}(1-\epsilon)
}{\sqrt{t_{\mathrm{mix}}^{(n)}t_{\mathrm{rel}}^{(n)}|\log \epsilon
|}} = O(1),  \text{ for any }0<\epsilon \le 1/4 .
 \end{equation}
 The details of the derivation of (\ref{eq: subgausiantrees}) from Lemma \ref{lem: gaussiantails} are left to the reader.
\label{sec: refined}
\begin{proposition}
\label{prop: Kac2}
Let $(\Omega,P,\pi)$ be a finite irreducible reversible Markov chain. Let $0<\epsilon < 1$. Let
$A \varsubsetneq\ \Omega$ be such that $\pi(A) \ge 1-\epsilon$. Denote the complement of $A$ by $B$. Denote $p:=1-\frac{1-\epsilon}{t_{\mathrm{rel}}}$ and $a:=\mathbb{E}_{\psi_{B}}[T_{A}]$. 
Let $z>1$ be such that $2p(z-1) \le 1-p$. Then
\begin{equation}
\label{eq: Kac4}
\max( \mathbb{E}_{\psi_{B}}[z^{T_{A}-\mathbb{E}_{\psi_{B}}[T_{A}]}], \mathbb{E}_{\psi_{B}}[z^{\mathbb{E}_{\psi_{B}}[T_{A}]-T_{A}}]) \le \exp
\left[   \frac{ 2a(z-1)^{2}}{1-p} \right].
\end{equation}
\end{proposition}  
\begin{proof}
   By (\ref{eq: Kac}) and (\ref{eq: CM1})
\begin{equation}
\label{eq: Kac5}
\begin{split}
& \mathbb{E}_{\psi_{B}}[z^{T_{A}}]=\sum_{k \ge 1}z^k \Pr_{\psi_{B}}[T_{A}=k]=1+(z-1)\sum_{k \ge 1}z^{k-1} \Pr_{\psi_{B}}[T_{A}\ge k] \\ & =1+(z-1)a\sum_{k
\ge 1}z^{k-1} \Pr_{\pi_{B}}[T_{A}= k] \leq 1+(z-1)a\sum_{k
\ge 1}(1-p)(pz)^{k-1} \\ & = 1+\frac{(z-1)a(1-p)}{1-pz}=1+(z-1)a \left(1+\frac{p(z-1)}{1-pz} \right) =1+(z-1)a
\left(1+\frac{ \frac{p(z-1)}{1-p}}{1-\frac{p(z-1)}{1-p}} \right) \\ & \le 1+(z-1)a
\left(1+\frac{2p(z-1)}{1-p} \right) \le \exp[a(z-1)+\frac{2ap(z-1)^{2}}{1-p}] , 
\end{split}
\end{equation}
where in the penultimate inequality we have used the assumption that $2p(z-1) \le 1-p $. We also have that
\begin{equation}
\label{eq: Kac6}
\begin{split}
& z^{-\mathbb{E}_{\psi_{B}}[T_{A}]
} \le \left(1-(z-1)+(z-1)^2\right)^{a} \le \exp[-a(z-1)+a(z-1)^{2}].
\end{split}
\end{equation}
Thus $\mathbb{E}_{\psi_{B}}[z^{T_{A}-\mathbb{E}_{\psi_{B}}[T_{A}]}] \le \exp \left[  a(z-1)^{2} \left(1+\frac{2p}{1-p}\right) \right] \le \exp
\left[   \frac{ 2a(z-1)^{2}}{1-p} \right].   $

Similarly,
\begin{equation}
\label{eq: Kac7}
\begin{split}
& \mathbb{E}_{\psi_{B}}[z^{-T_{A}}]=\sum_{k \ge 1}z^{-k} \Pr_{\psi_{B}}[T_{A}=k]=1-(1-z^{-1}) \sum_{k
\ge 1}z^{-(k-1)} \Pr_{\psi_{B}}[T_{A}\ge k] \\ & =1-(1-z^{-1})a\sum_{k
\ge 1}z^{-(k-1)} \Pr_{\pi_{B}}[T_{A}= k]=1-(1-z^{-1})a\sum_{k
\ge 1}(1-p)(p/z)^{k-1} \\ & = 1-\frac{(1-z^{-1})a(1-p)}{1-p/z}=1-(1-z^{-1})a \left(1-\frac{p(1-z^{-1})}{1-p/z}
\right) \\ & =1-(1-z^{-1})a
\left(1-\frac{ \frac{p(1-z^{-1})}{1-p}}{1-\frac{p(1-z^{-1})}{1-p}} \right)  \le1-(1-z^{-1})a
\left(1-\frac{2p(1-z^{-1})}{1-p} \right) \\ & \le \exp \left[-a(1-z^{-1})+\frac{2ap(z-1)^{2}}{1-p}\right].
\end{split}
\end{equation}
We also have that $z^{\mathbb{E}_{\psi_{B}}[T_{A}]
} \le \left(1+(z-1)\right)^{a} \le e^{a(z-1)}$. Note that $a(z-1)-a(1-z^{-1})=a(z-1)^{2}/z \le a(z-1)^2$. Hence  $\mathbb{E}_{\psi_{B}}[z^{\mathbb{E}_{\psi_{B}}[T_{A}]-T_{A}}] \le \exp
\left[  a(z-1)^{2} \left(1+\frac{2p}{1-p}\right) \right] \le \exp
\left[   \frac{ 2a(z-1)^{2}}{1-p} \right]$.
\end{proof}
\begin{lemma}
\label{lem: gaussiantails}
Let $(V,P,\pi)$ be a Markov chain on a tree $(T,o)$. Let $x,y \in V$ be such that $y \prec x$. Denote
$t_{x,y}:=\mathbb{E}_x[T_{y}] $ and  $b=b_{x,y}:=\sqrt{t_{x,y}t_{\mathrm{rel}}}$. Then
\begin{equation}
\label{eq: gaussian1}
\Pr_x[T_{y}-t_{x,y} \ge cb ] \vee \Pr_x[t_{x,y} -T_{y} \ge cb ] \le e^{-c^2/20}, \text{ for any } 0 \le c \le \frac{5}{2} \sqrt{t_{x,y}/t_{\mathrm{rel}}}.
\end{equation}
\end{lemma}
\begin{proof}
Let $(v_0=x,v_1,\ldots,v_{k}=y)$ be the path from $x$ to $y$.
Define $\tau_i:=T_{v_{i}}-T_{v_{i-1}}$. Then by the tree structure, under
$\Pr_x$, we have that $T_{y}=\sum_{i=1}^{k}\tau_i$ and that $\tau_1,\ldots,\tau_k$
are independent. Denote $p:=1-\frac{1}{2t_{\mathrm{rel}}}$.  Denote $a_i:=\mathbb{E}_{x}[\tau_i] $. Fix some $0 \le c \le
\frac{5}{2} \sqrt{t_{x,y}/t_{\mathrm{rel}}}$. 
Set $z_{c}=z_{c,x}:=1+\frac{c}{10b}$. Note that $2p(z_{c}-1) \le \frac{c}{5b} \le \frac{1}{2t_{\mathrm{rel}}} =1-p$.  Then by (\ref{eq: Kac4})
\begin{equation}
\label{eq: strongconcentration1}
\begin{split}
& \Pr_x[T_{y}-t_{x,y} \ge cb ]=\Pr_x[z_{c}^{T_{y}-t_{x,y} } \ge z_{c}^{cb} ] \le \mathbb{E}_x[z_{c}^{T_{y}-t_{x,y}} ]z_{c}^{-cb}=z_{c}^{-cb} \prod_{i=1}^k  \mathbb{E}_x[z_{c}^{\tau_i-a_i
} ] \\ &  \le \exp [(-(z_{c}-1)+(z_{c}-1)^{2})cb]\prod_{i=1}^k  \exp
\left[   \frac{ 2a_{i}(z_{c} -1)^{2}}{1-p} \right] \\ & = \exp \left[-\frac{c^{2}}{10}+\frac{c^3}{100b} \right]\exp \left[\frac{2t_{\mathrm{rel}}t_{x}c^2}{50b^2}\right] \le \exp \left[-\frac{c^{2}}{10}+\frac{c^3}{100b} + \frac{c^2}{25} \right] \le e^{-c^2/20}.
\end{split}
\end{equation}
The inequality $\Pr_x[t_{x,y}
-T_{y} \ge cb ] \le e^{-c^2/20}$ is proved in an analogous manner.
\end{proof}
\section{Weighted random walks on the interval with bounded jumps}
In this section we prove Theorem \ref{thm: semibd} and establish that product condition is sufficient for cutoff for a sequence of $(\delta,r)$-SBD chains. Although we think of $\delta$ as being
bounded away from 0, and of $r$ as a constant integer,
it will be clear that our analysis remains valid as long as $\delta$ does not tend to 0, nor does $r$  to infinity, too rapidly  in terms of some functions of $t_{\mathrm{rel}}/t_{\mathrm{mix}}$.

Throughout the section, we use $C_1,C_2,\ldots$ to describe positive
constants which depend only on $\delta$ and $r$. Consider a $(\delta,r)$-SBD chain on $([n],P,\pi)$. We call a state $i \in [n]$ a central-vertex if $\pi([i-1]) \vee \pi([n]\setminus [i])\le 1/2$. As opposed to the setting of Section \ref{sec: trees}, the sets $[i-1]$ and $[n]
\setminus [i]$ need not be connected components of $[n] \setminus \{i\}$ w.r.t.~the chain, in the sense that it might be possible for the chain to get from $[i-1]$ to $[n] \setminus [i]$ without first hitting $i$ (skipping over $i$).  We pick a central-vertex $o$ and call it the root. 


Divide $[n]$ into $m:=\lceil n/r \rceil$ consecutive disjoint intervals, $I_1,\ldots,I_m$ each of size $r$, apart from perhaps $I_m$. We call each such interval a \emph{block}. Denote by $I_{\tilde{o}}$ the unique block such that the root $o$ belongs to it. Since we are assuming the product condition, in the setup of Theorem \ref{thm: semibd} we can assume without loss of generality that $I_{\tilde{o}}\neq [n]$. Observe the following. Suppose $v\notin I_{\tilde{o}}$ is a neighbour of $I_{\tilde{o}}$ in $[n]$. Then by reversibility and the definition of a $(\delta, r)$ chain, we have for all $v'\in I_{\tilde{o}}$, $\pi(v)\geq \delta^r \pi(v')$. Hence $\pi(I_{\tilde{o}})\leq \frac{r}{r+\delta^r}$. For the rest of this section let us fix $\alpha=\alpha(\delta,r)=1-\frac{\delta^r}{4(r+\delta^r)}$.  


Recall that in Section \ref{sec: trees} we exploited the tree structure to reduce the problem of showing cutoff to showing the concentration of the hitting time of the central vertex by showing that starting from the central vertex the chain hits any large set quickly. We argue similarly in this case with central vertex replaced by the central block. First we need the following lemma.
\begin{lemma}
\label{lem: bdlem1'}
In the above setup, let $I:=\{v,v+1,\ldots,v+r-1\} \subset [n] $. Let $\mu \in \mathscr{P} (I)$. Then
\begin{equation}
\label{eq: 7.1.1}
\mathbb{E}_{\mu}[T_{A}] \le 
\max_{y \in I} \mathbb{E}_{y}[T_A] \le \delta^{-r} \min_{ x \in I }\mathbb{E}_{x}[T_{A}], \text{ for any }A \subset \Omega \setminus I.
\end{equation}
Consequently, for any $i \in I$ and $A \subset [v-1]$ (resp.~$A \subset [n] \setminus [v+r-1]$) we have that
\begin{equation}
\label{eq: 7.1.2}
\mathbb{E}_{i}[T_A] \le \delta^{-r}\mathbb{E}_{\pi_{[n]\setminus[v-1]}}[T_A], \text{ (resp.~}\mathbb{E}_i[T_A]\le \delta^{-r}\mathbb{E}_{\pi_{[v+r-1]}}[T_A]).
\end{equation}
\end{lemma}
\begin{proof}
We first note that (\ref{eq: 7.1.2}) follows from (\ref{eq: 7.1.1}). Indeed, by condition
(i) of the definition of a $(\delta,r)$-SBD chain, if  $A \subset [v-1]$ (resp.~$A \subset [n]
\setminus [v+r-1]$), then under $\Pr_{\pi_{[n]\setminus [v-1]}}$ (resp.~under $\Pr_{\pi_{[v+r-1]}}$), $T_I \le T_A $. Thus  (\ref{eq: 7.1.2}) follows from (\ref{eq: 7.1.1}) by averaging over $X_{T_{I}}$. We now prove (\ref{eq: 7.1.1}). 

Fix some $A$ such that $A \subset [n] \setminus I $. Fix some distinct $x,y \in I$. Let $B_1$ be the event that $T_y \le T_{A}$. One way in which $B_1$ can occur is that the chain would move from $x$ to $y$ in $|y-x|$ steps such that $|X_k-X_{k-1}|=1$ for all $1 \le k \le |y-x|$. Denote the last event by $B_2$. Then $$\mathbb{E}_x[T_A] \ge \mathbb{E}_x[T_A1_{B_{2}}] \ge
\Pr[B_2] \mathbb{E}_y[T_A] \ge \delta^r \mathbb{E}_y[T_A].$$
Minimizing over $x$ yields that for any $y \in I$ we have that $\mathbb{E}_y[T_A] \le \delta^{-r} \min_{ x \in I }\mathbb{E}_{x}[T_{A}]$, from which (\ref{eq: 7.1.1}) follows easily. 
\end{proof}

The next proposition reduces the question of proving cutoff for a sequence of $(\delta,r)$-SBD chains under the product condition to that of showing an appropriate concentration for the hitting time of central block. The argument is analogous to the one in Section \ref{sec: trees} and hence we only provide a sketch to avoid repititions. As in Section \ref{sec: trees}, for $\epsilon\in (0,1)$ let $\tau_{C}(\epsilon)=\min\{t:\Pr_x[T_{I_{\tilde{o}}}>t]\leq \epsilon \text{ } \forall x\in [n]\}$. 

\begin{proposition}
\label{p:sbdreduction}
In the above set-up, suppose there exists universal constants $C_{\epsilon}$ for $\epsilon\in (0,\frac{1}{8})$ and a constant $w_n$ depending on the chain such that we have
\begin{equation}
\label{e:sbdreduce1}
\tau_C(\epsilon)-\tau_C(1-\epsilon)\leq C_{\epsilon}w_n~\text{for all}~\epsilon\in (0,\frac{1}{8}).
\end{equation}
Then we have for some unversal constants $C'_{\epsilon},C''_{\epsilon}$ and for all $\epsilon\in (0,1/8)$ 
\begin{equation}
\label{e:sbdreduce2}
\hit_{\alpha}(3\epsilon/2)-\hit_{1/2}(1-3\epsilon/2)\leq C_{\epsilon}w_n+ C'_{\epsilon}t_{\rm rel}~\text{and}
\end{equation}
\begin{equation}
\label{e:sbdreduce3}
t_{\rm mix}(2\epsilon)-t_{\rm mix}(1-2\epsilon)\leq C''_{\epsilon}(w_n+t_{\rm rel}).
\end{equation} 
\end{proposition}

\begin{proof}
Observe that (\ref{e:sbdreduce3}) follows from (\ref{e:sbdreduce2}) using Proposition \ref{prop: TVbound0} and Corollary \ref{prop: hitpqinequalities}. To deduce (\ref{e:sbdreduce2}) from (\ref{e:sbdreduce1}), we argue as in Lemma \ref{lem: TVboundstrees} using Lemma \ref{cor: 7.2} below which shows that starting from any vertex in $I_{\tilde{o}}$ the chain hits any set of $\pi$-measure at least $\alpha$ in time proportional to $t_{\rm rel}$ with large probability. We omit the details.  
\end{proof}

\begin{lemma}
\label{cor: 7.2}
Let $v \in I_{\tilde{o}}$. Let $C \subset [n] $ be such that $\pi(C) \ge \alpha$.  Then $\mathbb{E}_v[T_C] \le C(\alpha)\delta^{-r}t_{\mathrm{rel}}$ for some constant $C(\alpha)$. In particular, $\mathrm{hit}_{\alpha,v}(\epsilon) \le \epsilon^{-1}C(\alpha)\delta^{-r}t_{\mathrm{rel}}$ by Markov inequality.
\end{lemma}
\begin{proof}
Let $I_{\tilde{o}}=\{v_1,v_1+1,\ldots, v_2\}$. Set $A_1=[v_1-1]$ and $A_2=[n]\setminus [v_2]$. For $i=1,2$, let $C_i=C\cap A_i$. Using the definition of $\alpha$ without loss of generality let $\pi(C_1)\geq \frac{1-\alpha}{2}$. Set $A=A_2\cup I_{\tilde{o}}$. By (\ref{eq: 7.1.2}) 
$$\mathbb{E}_v[T_C]\leq \mathbb{E}_v[T_{C_1}]\leq \delta^{-r}\mathbb{E}_{\pi_A}[T_{C_1}].$$
The proof is completed by observing that $\pi(A)\geq \frac{1}{2}$ and using Lemma \ref{lem: AF1}.
%
\end{proof}


Observe that, arguing as in Corollary \ref{cor: concentrationfortrees1}, it follows using Cheybeshev inequality that (\ref{e:sbdreduce1}) holds for some constants $C_{\epsilon}$ if we take $w_n=\max_{x\in [n]}\sqrt{ \Var_{x}[T_{I_{\tilde{o}}}]}$. Theorem \ref{thm: semibd} therefore follows at once from Proposition \ref{p:sbdreduction} provided we establish $\Var_{x}[T_{I_{\tilde{o}}}] \leq C_1\mathbb{E}_{x}[T_{I_{\tilde{o}}}]t_{\rm rel}$ for all $x\notin I_{\tilde{o}}$ (since $E_x[T_{I_{\tilde{o}}}]=O(t_{\rm mix})$). This is what we shall do. 

Observe that the root induces a partial order on the blocks. We say that $I_j \prec I_k$ if $I_j$ is a block between $I_k$ and $I_{\tilde{o}}$. For $j\in [m]$, $I_j\neq I_{\tilde{o}}$,  we define the parent block of $I_j$ in the obvious manner and denote its index by $f_j$. We define  $$T(j):=T_{I_j} \text{ and } \bar \tau_j :=T(f_{j})-T(j).$$ As mentioned above, for $I_j\neq I_{\tilde{o}}$ and $x\in I_j$ arbitrary we will bound $\Var_{x}[\sum \bar \tau_\ell]$, where $x \in I_j$ is arbitrary, and the sum is taken over blocks between $I_{j}$ and $I_{\tilde{o}} $. As opposed to the situation in Section \ref{sec: trees}, the terms in the sum are no longer independent. We now show that the correlation between them decays exponentially (Lemma \ref{lem: bdlem2}) and that for all $\ell$ we have that $\Var_{x}[\bar \tau_\ell] \le C_{2}t_{\mathrm{rel}}\mathbb{E}_{x}[\bar \tau_\ell]$ (Lemma \ref{lem: Kac3}). This shall establish the necessary upper bound mentioned above. We
omit the details. 
\begin{lemma}
\label{lem: bdlem1}
In the above setup, let $v \in [m] \setminus \{o\}$  Let
$(v_0=v,v_1,\ldots,v_{s})$ be indices of consecutive blocks. Let $\mu_1, \mu_2 \in
\mathscr{P}(I_v)$.  Let $k \in [s]$. Denote by $\nu_{k}^{(j)}$ ($j=1,2$)
the hitting
distribution of $I_{v_k}$ starting from initial distribution $\mu_j$ (i.e.~$\nu_{k}^{(j)}(z):=\Pr_{\mu_j}[X_{T(v_k)}=z]
$). Then $\|\nu_{k}^{(1)}-\nu_{k}^{(2)} \|_{\mathrm{TV}} \le (1-\delta^{r})^{k}
$. 
\end{lemma}
\begin{proof}
It suffices to prove the case $k=1$ as the general case follows by induction
using the Markov property. The case $k=1$ follows from coupling the chain
with the two different starting distributions in a way that with probability
at least $\delta^r$ there exists some $z_{v} \in I_v$ such that both chains hit $z_v$ before hitting $I_{f_{v}}$ and
from that moment on they follow the same trajectory. The fact that the hitting
time of $z_{v}$
might be different for the two chains makes no difference. We now describe
this coupling more precisely.

Let $\mu_1, \mu_2 \in
\mathscr{P}(I_v)$. There exists a coupling $(X_t^{(1)},X_t^{(2)})_{t \ge
0}$ in which $(X_t^{(i)})_{t \ge 0}$ is distributed as the chain $(\Omega,P,\pi)$
with initial distribution $\mu_i$ ($i=1,2$), such that $\Pr_{\mu_1,\mu_2}[S]
\ge \delta$, where $\Pr_{\mu_1,\mu_2}$ is the corresponding probability measure and the event $S$ is defined as follows. Let $R:=\min
\{t:X_{t}^{(1)} =X_{0}^{(2)} \}$ and $L_{i}:=\min \{t:X_{t}^{(i)} \in I_{f_{v}}\}$. Let $S$ denote the event: $R \le L_{1}$ and $X_{R+t}^{(1)}=X_{t}^{(2)}$ for any $t \ge 0$. Note that on $S$,
$X_{L_{1}}^{(1)}=X_{L_{2}}^{(2)}$.  
  Hence for any $D \subset I_{v_k}$,
\begin{equation*}
\begin{split}
&\nu_{1}^{(1)}(D)-\nu_{1}^{(2)}(D) = \Pr_{\mu_1,\mu_2}[X_{L_1}^{(1)}
\in D ]- \Pr_{\mu_1,\mu_2}[X_{L_{2}}^{(2)}\in D ] \\ & \le \Pr_{\mu_1,\mu_2}[X_{L_{1}}^{(1)}\in
D ,X_{L_{2}}^{(2)} \notin D]   \le 1-\Pr_{\mu_1,\mu_2}[S] \le 1-\delta^{r}.
\end{split}
\end{equation*}
\end{proof}

\begin{lemma}
\label{lem: bdlem2}
In the setup of Lemma \ref{lem: bdlem1}, let $0 \le i<j < s$.
Let $\mu \in \mathscr{P}(I_v)$. Write $\tau_i:=\bar \tau_{v_{i}}$ and $\tau_j:= \bar \tau_{v_{j}}$.
Then
$$\mathbb{E}_{\mu}[\tau_i \tau_j] \le \mathbb{E}_{\mu}[\tau_i]\mathbb{E}_{\mu}[
\tau_j]\biggl(1+(1-\delta^{r})^{j-i-1}\delta^{-r}\biggr).$$ 
\end{lemma}
\begin{proof}
Let $\mu_{i+1}$ and $\mu_j$ be the hitting distributions of $I_{v_{i+1}}$
and
of $I_{v_{j}}$, respectively, of the chain with initial distribution $\mu$.
Note that $\mathbb{E}_{\mu}[\tau_j]=\mathbb{E}_{\mu_{i+1}}[\tau_j]=\mathbb{E}_{\mu_{j}}[\tau_j]$. Clearly
\begin{equation}
\label{eq: mixedterm}
\mathbb{E}_{\mu}[\tau_i \tau_j] \le \mathbb{E}_{\mu}[\tau_i]\max_{y
\in I_{v_{i+1}}}\mathbb{E}_{y}[\tau_j].
\end{equation}
 Let $y^{*} \in I_{v_{i+1}}$ be the state achieving the maximum in the RHS
above.
By Lemma \ref{lem: bdlem1} we can couple successfully the hitting distribution
of $I_{v_{j}}$  of the chain started from $y^{*}$ with that of the chain starting from initial distribution $\mu_{i+1}$ with probability at least $1-(1-\delta^{r})^{j-i-1}
$. The latter distribution is simply $\mu_{j}$. If the coupling fails, then
by (\ref{eq: 7.1.1}) we can upper bound the conditional expectation
of $\tau_j$ by $\delta^{-r}\E_{\mu}[\tau_j]$. Hence
$$\mathbb{E}_{y^{*}}[\tau_j] \le  \mathbb{E}_{\mu_j}[\bar \tau_j]+(1-\delta)^{j-i-1}
\delta^{-r}\mathbb{E}_{\mu}[\tau_j] =\mathbb{E}_{\mu}[\tau_j]\biggl(1+(1-\delta^{r})^{j-i-1}
\delta^{-r}\biggr).$$
The assertion of the lemma follows by plugging this estimate in (\ref{eq:
mixedterm}).
\end{proof}
\begin{lemma}
\label{lem: Kac3}
Let $j \in [m] \setminus \{o\} $. Let $\nu \in \mathscr{P}([n])$. Then there exists some $C_{1},C_{2}>0$ such that $\mathbb{E}_{\nu}[\bar \tau_j^2] \le C_{1} t_{\mathrm{rel}}\Phi(A_{j}) \le C_{2}t_{\mathrm{rel}}\mathbb{E}_{\nu}[\bar \tau_j]$.
\end{lemma}
\begin{proof}
Let $\mu:=\psi_{A_{j}}$. By condition (i) in the definition of a $(\delta,r)$-SBD chain, $\mu \in \mathscr{P}(I_{j})$. By (\ref{eq: kac1}), $\mathbb{E}_{\mu}[\bar \tau_j^2]
\le C_{3} t_{\mathrm{rel}}\Phi(A_{j}) \le C_4t_{\mathrm{rel}}\mathbb{E}_{\mu}[\bar \tau_j] $. The proof is concluded using the same reasoning as in the proof of (\ref{eq: 7.1.1}) to argue that the first and second moments of $\bar \tau_j$ w.r.t.~different initial distributions can change by at most some multiplicative constant.
\end{proof}

\section{Examples}
\label{sec: examples}

\subsection{Aldous' example}
We now present a small variation of Aldous' example (see $\cite{levin2009markov}$,
Chapter 18) of a sequence of chains which satisfies the product condition
but does not exhibit cutoff. This example demonstrates that Theorem \ref{thm: semibd} may fail if condition (ii) in the definition of a $(\delta,r)$-semi birth and death chain is not satisfied. The main point in the construction is that the hitting times of worst sets are not concentrated.

\begin{figure*}[!h]
\begin{center}
\includegraphics[height=10cm,width=17.1cm]{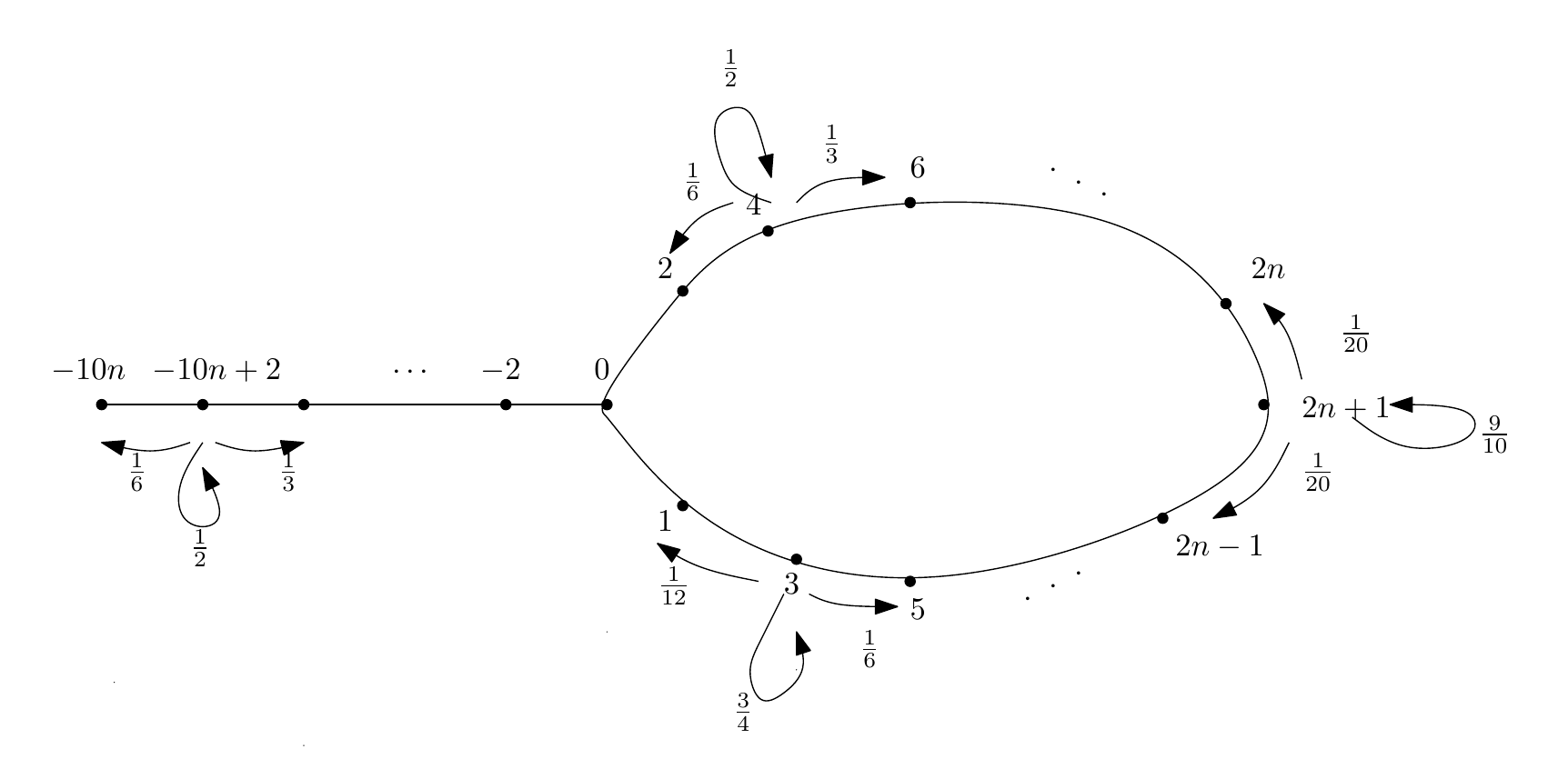}
\caption{We consider a Markov Chain on the above graph with the following transition probabilities:$P_{n}(x,x)=1/2$ for $x$ even and $P_{n}(x,x)=3/4$ for $x$ odd. $P_n(0,2)=P_n(0,1)=\frac{1}{5}, P_n(0,-2)=\frac{1}{10}$, $P_n(-10n, 10n+2)=1/2$, $P_n(2n+1,2n)=P_n(2n+1,2n-1)=\frac{1}{20}$. All other transition probabilities are given by: $P_n(2i,\min \{2i+2,2n+1\})=\frac{1}{3}$, $P_n(2i,2i-2)=P_n(2i-1,2i+1)=\frac{1}{6}$, $P_n(2i-1,\max \{ 2i-3,0\})=\frac{1}{12}$.}
\label{f:Aldousexample}
\end{center}
\end{figure*}

\begin{example}
Consider the chain $(\Omega_n,P_n,\pi_n)$, where  $\Omega_{n}:=\{-10n,-10n+2,\ldots,-2,0\} \cup [2n+1]$. Think of $\Omega$ as two paths (we call them branches) of length $n$ joined together at the ends and a path of length $5n$ joined to them at $0$ (see Figure \ref{f:Aldousexample}). Set $P_{n}(x,x)=1/2$ if $x$ is even, $P_{n}(x,x)=3/4$ if $x$ is odd and $x<2n+1$ and $P_n(2n+1,2n+1)=9/10$.

Conditionally on not making a lazy step the walk moves with a fixed bias towards $2n+1$ (apart from at the states $-10n,0,2n+1$): 
$$P_n(2i,\min \{2i+2,2n+1\})=2P_n(2i,2i-2)=2P_n(2i-1,2i+1)=4P_n(2i-1,\max
\{ 2i-3,0\})=\frac{1}{3}.$$  Finally, we set $P_{n}(-10n,-10n+2)=1/2$, $P_n(0,2)=P_n(0,1)=2P_n(0,-2)=\frac{1}{5}$ and $P_n(2n+1,2n)= P_n(2n+1,2n-1)=\frac{1}{20}$. It is easy to check that this chain is indeed reversible.

By Cheeger inequality (e.g.~\cite{levin2009markov}, Theorem 13.14), $t_\mathrm{rel}^{(n)}=O(1)$, as the bottleneck-ratio is bounded from below. In particular, the product condition holds.  As $\pi_n(2n+1)>1/2$, there is $\mathrm{hit}_{1/2}$-cutoff iff starting from $-10n$, the hitting-time of $2n+1$ is concentrated. We now explain why this is not the case. In particular, by Theorem \ref{thm: psigmacutoffequiv}, there is no cutoff.

\begin{figure*}[!h]
\begin{center}
\includegraphics[height=5cm,width=7cm]{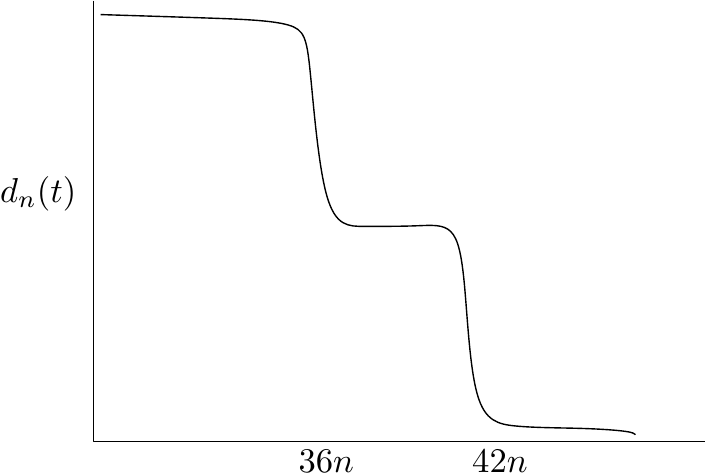}
\caption{Decay in total variation distance for Aldous' example: it does not have cutoff}
\label{f:dtjump}
\end{center}
\end{figure*}

Let $Y$ denote the last step away from $0$ before $T_{2n+1}$. Observe that if $Y=2$ (respectively, $Y=1$), then the chain had to reach $2n+1$ through the path $(2,4,\ldots,2n)$ ($(1,3,\ldots,2n-1)$, respectively). Denote, $Z_{i}:=T_{2n}1_{Y=i}$, $i=1,2$. Then on $Y=i$,  $T_{2n}=Z_i$, and its conditional distribution is concentrated around $42n$ for $i=1$ and around $36n$ for $i=2$, with deviations of order $\sqrt{n}$ . Since both $Y=1$ and $Y=2$ have probability bounded away from $0$, it follows that $d_{n}(37n)$ and $d_{n}(41n)$ are both bounded away from 0 and 1 (see Figure \ref{f:dtjump}). In particular, the product condition holds but there is no cutoff.
\end{example}
\subsection{Sharpness of Theorem \ref{thm: psigmacutoffequiv}}
Now we give an example to show that in Proposition \ref{prop: prodcondandhitcutoff} (and hence in Theorem \ref{thm: psigmacutoffequiv})  the value $\frac{1}{2}$ cannot be replaced by any larger value.

\begin{figure*}[h]
\begin{center}
\includegraphics[height=9cm,width=16cm]{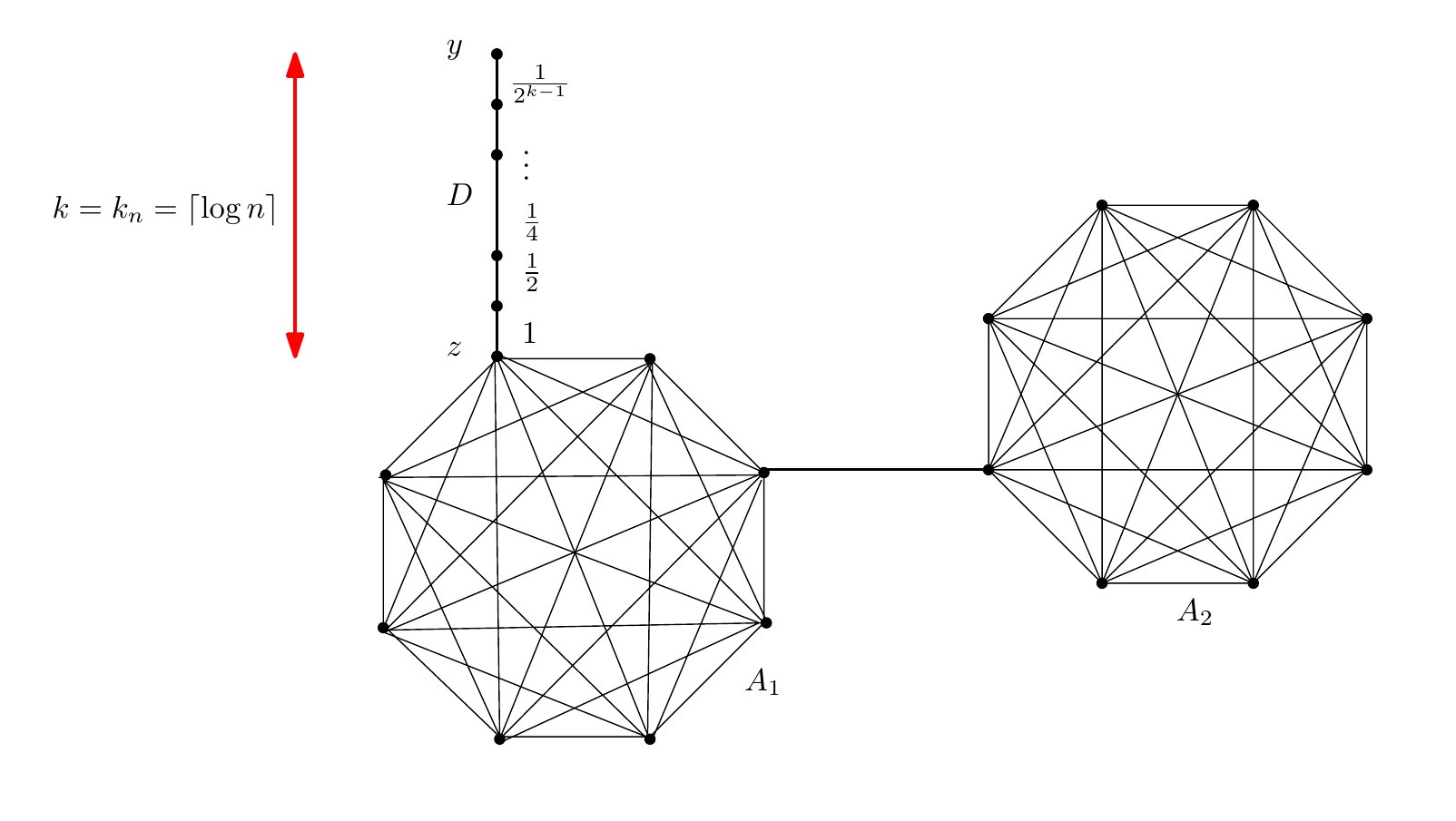}
\caption{We consider a lazy weighted nearest-neighbor random walk on the above
graph consisting of two disjoint cliques $A_1$ and $A_2$ of size $n$ connected by a single
edge and a path of length $k_n=\lceil \log n \rceil$ connected to $A_1$. The edge weights of all edges incident to vertices in $A_1\cup A_2$ is 1, while those belonging to the path are indicated in the figure. Inside
the path, the walk has a fixed bias towards the clique.}
\label{f: 2cliques}
\end{center}
\end{figure*}

\begin{example}
\label{ex: Sharpness}
Let $(\Omega_n,P_n,\pi_n)$ be the nearest-neighbor weighted random walk from Figure \ref{f: 2cliques}. Then $\rel^{(n)}=\Theta ( \mix^{(n)}) $, yet for every $1/2<\alpha <1$, the sequence exhibits $\hit_{\alpha}$-cutoff.
\end{example}
\begin{proof}
Let $\Phi_{n}:=\min_{A \subset \Omega_n:0<\pi(A) \le 1/2}\Phi_{n}(A)$ be the Cheeger constant of the $n$-th chain, where
$\Phi_{n}(A):= \frac{\sum_{a
\in A,b \in A^c}\pi_{n}(a)P_{n}(a,b)}{\pi_{n}(A)}$. Then by taking $A$ to be either $A_1$ or $A_2$, by Cheeger inequality (e.g.~\cite{levin2009markov}, Theorem 13.14), we have that $\rel^{(n)} \ge \frac{1}{2\Phi_{n}} \ge c_1 n^2 \ge c_2 \mix $. By Fact \ref{fact: cutoffandtrel}, indeed  $\rel^{(n)}=\Theta ( \mix^{(n)}) $.

Fix some $1/2<\alpha<1$. Let $B \subset \Omega_n$ be such that $\pi(B) \ge \alpha$. Denote the set of vertices belonging to the path, but not to $A_1$ by $D$. Then $\pi_n(D)=O(n^{-2})=o(1)$. Consequently, $\pi(A_i \cap B) \ge \alpha-1/2-o(1) $, for $i=1,2$. Using this observation, it is easy to verify that for all $x \in A_1 \cup A_2 $ we have that
\begin{equation}
\label{eq: sharpnessex1}
\mathrm{hit}_{\alpha,x}(\epsilon) \le c_{\alpha} \log (1/\epsilon), \text{ for any }0<\epsilon <1,
\end{equation}
for some constant $c_{\alpha}$ independent of $n$.

Let $y$ be the endpoint of the path which does not lie in $A_1$. Let $z$ be the other endpoint of the path. The hitting time of $z$ under $\Pr_{y}$ is concentrated around time $6 \log n $. Then by (\ref{eq: sharpnessex1}), together with the Markov property (using the same reasoning as in the proof of Lemma \ref{lem: TVboundstrees}) for all sufficiently large $n$ we have that for any $0<\epsilon \le 1/4$
\begin{equation}
\label{eq: sharpnessex2}
\begin{split}
\mathrm{hit}_{\alpha,y}^{(n)}(2\epsilon)& \le (6+o(1)) \log n + \mathrm{hit}_{\alpha,z}^{(n)}(\epsilon) =(6+o(1)) \log n,      
\\ \mathrm{hit}_{\alpha,y}^{(n)}(1-\epsilon) & \ge (6-o(1)) \log n.
\end{split}
\end{equation}
Similarly to the proof of Lemma \ref{lem: TVboundstrees}, for any $B \subset \Omega_n $ and any $x \in D$, we have that $\Pr_{y}[T_{B \setminus D}>t] \ge\Pr_{x}[T_{B
 }>t]  $, for all $t$.
Since $\pi_n(D)=o(1)$, this implies that for all sufficiently large $n$, for any $1/2<\alpha<1$, there exists some $1/2<\alpha'<\alpha$ ($\alpha'$ depends on $\alpha$ but not on $n$), such that  for any $x \in D$ we have that $\mathrm{hit}_{\alpha,y}^{(n)}(\epsilon) \ge \mathrm{hit}_{\alpha',x}^{(n)}(\epsilon)   $, for all $0<\epsilon<1$. This, together with (\ref{eq: sharpnessex1}) and the fact that the leftmost terms in both lines of (\ref{eq: sharpnessex2}) are up to negligible terms independent of $\alpha$ and $\epsilon$,   implies that the sequence of chains exhibits $\hit_{\alpha}$-cutoff for all $1/2<\alpha<1$.
\end{proof}
\begin{remark}
One can modify the sequence from Example \ref{ex: Sharpness} into a sequence of lazy simple nearest-neighbor random walks on a graph. Construct the $n$-th graph in the sequence as follows. Start with a binary tree $T$ of depth $n$. Denote its root by $y$, the set of its leaves by $A_1$ and $D:=T \setminus A_1$. Turn $A_1$ into a clique by connecting every two leaves of $T$ by an edge. Take another disjoint complete graph of size $|A_1|=2^{n}$ and denote its vertices by $A_2$. Finally, connect $A_1$ and $A_2$ by a single edge. Since the number of edges which are incident to $D$ is at most $2^{n+2}$, while the total number of edges of the graph is greater than $2^{2n}$, we have that $\pi_n(D)=o(1)$. The analysis above can be extended to this example with minor adaptations (although a rigorous analysis of this example is somewhat more tedious).
\end{remark}
\vspace{2mm}

\section*{Acknowledgements}
We are grateful to David Aldous, Allan Sly, Perla Sousi and Prasad Tetali for many helpful suggestions.

...

\nocite{*}
\bibliographystyle{plain}
\bibliography{cutoff}
\end{document}